\documentclass[12pt]{amsart}

\usepackage[USenglish]{babel}
\usepackage{amsmath,amsthm,amssymb,amsfonts}
\usepackage{mathtools}
\usepackage{mathrsfs}
\usepackage{bbm}

\usepackage{indentfirst}
\usepackage{microtype}
\usepackage{enumerate}
\usepackage{accents}
\usepackage{cancel} %cancel arrows
\usepackage{xcolor}

\usepackage{multirow}
\usepackage{longtable}

\usepackage{float}
\usepackage{graphicx}
\usepackage{placeins}
\usepackage{caption}

\usepackage{tikz}
\usetikzlibrary{patterns}
\usepackage{tikz-cd}

\usepackage{hyperref}

\allowdisplaybreaks

\newtheorem{thm}{Theorem}[section]

\newtheorem{lem}[thm]{Lemma}
\newtheorem{prop}[thm]{Proposition}

\theoremstyle{definition}

\newtheorem*{defx}{Definition}

\theoremstyle{remark}
\newtheorem{rem}[thm]{Remark}
\newtheorem*{xrem}{Remark}
\newtheorem*{ntt}{Notation}

\numberwithin{equation}{section}

\newcommand{\N}{\mathbb{N}}      % N = Naturals
\newcommand{\Z}{\mathbb{Z}}      % Z = Integers
\newcommand{\Q}{\mathbb{Q}}      % Q = Rationals
\newcommand{\R}{\mathbb{R}}      % R = Reals
\newcommand{\eps}{\varepsilon}   % epsilon
\newcommand\restr[2]{{           % restriction of function
  \left.\kern-\nulldelimiterspace #1%
  \right|_{#2}%
 }}

\newcommand\minus{               % small minus
  \setbox0=\hbox{-}%
  \vcenter{%
    \hrule width\wd0 height \the\fontdimen8\textfont3%
  }%
}

\newcommand{\bigslant}[2]{       % Big Slant
  {\raisebox{.2em}{$#1$} \left/ \raisebox{-.2em}{$#2$} \right.}}

\newcommand{\Var}{\mathrm{Var}}  % Variance
\newcommand{\A}{\mathscr{A}}
\newcommand{\B}{\mathscr{B}}
\newcommand{\C}{\mathscr{C}}

\newcommand{\rhohat}{\widehat{\rho}}

\newcommand{\mupp}{\mathcal{M}^{*}}
\newcommand{\mlow}{\mathcal{M}_{*}}

\tikzset{                        % symbol instead of arrow
    symbol/.style={%
        ,draw=none
        ,every to/.append style={%
            edge node={node [sloped, allow upside down, auto=false]{$#1$}}}
    }
}

\newcommand{\uptxt}[2]{          % small text above
  \mathrel{\overset{\makebox[0pt]{\mbox{\normalfont\tiny #1}}}{#2}}}

\newcommand{\negphantom}[1]{\settowidth{\dimen0}{#1}\hspace*{-\dimen0}}

\newcommand{\corr}[1]{\textcolor{red}{#1}}

\restylefloat{table}

\begin{document}

\title{An extension of the Erd\H{o}s--Tetali theorem}%
\author{Christian T\'afula}%

\curraddr{Institute of Mathematics, Statistics\\
and Computer Science\\
University of S\~ao Paulo\\
Rua do Mat\~ao, 1010\\
S\~ao Paulo, SP 05508-090\\
Brazil}

\address{RIMS, Kyoto University, 606-8502 Kyoto, Japan}%
\email{tafula@ime.usp.br}%


% ----------------------------------------------------------------
\begin{abstract}
 Given a sequence $\A=\{a_0<a_1<a_2\ldots\}\subseteq \N$, let $r_{\A,h}(n)$ denote the number of ways $n$ can be written as the sum of $h$ elements of $\A$. Fixing $h\geq 2$, we show that if $f$ is a suitable real function (namely: locally integrable, $O$-regularly varying and of positive increase) satisfying {\color{red}
 \[ x^{1/h}\log(x)^{1/h} \ll f(x) \ll x^{1/(h-1) - \eps} \text{ for some } \eps > 0, \]}
 then there must exist $\A\subseteq\N$ with $|\A\cap [0,x]|=\Theta(f(x))$ for which $r_{\A,h+\ell}(n) = \Theta(f(n)^{h+\ell}/n)$ for all $\ell \geq 0$. Furthermore, for $h=2$ {\color{red} the same conclusion holds under} $x^{1/2}\log(x)^{1/2} \ll f(x) \ll x$.
 
 The proof is somewhat technical and the methods rely on ideas from regular variation theory, which are presented in an appendix with a view towards the general theory of additive bases. We also mention an application of these ideas to Schnirelmann's method.
 
 {\color{red} Corrections to the published version are highlighted in red.}
\end{abstract}
\maketitle
% ----------------------------------------------------------------

\section{Introduction}
 Denote by $\N$ the set of natural numbers with $0$. Given a sequence $\A=\{a_0<a_1<a_2\ldots\}\subseteq \N$ and an integer $h\geq 2$, the \emph{$h$-fold sumset} $h\A$ is the set $\{\sum_{i=1}^h k_i : k_1,\ldots,k_h \in \A \}$. We say that $\A$ is an \emph{additive basis} when there is $h\geq 2$ such that $\N\setminus h\A$ is empty, and the least such $h$ is the \emph{order} of $\A$. An \emph{$h$-basis} is an additive basis of order $h$. The \emph{representation functions} $r_{\A,h}(n)$ and $s_{\A,h}(x)$ count the number of solutions of $k_1+k_2+\ldots+k_h = n$ and $k_1+k_2+\ldots+k_h \leq x$ resp., where $h\geq 1$ is fixed and $k_i \in \A$, considering permutations in the sense of the formal series:
 \begin{equation}
  \left(\sum_{a\in \A} z^a\right)^h = \sum_{n\geq 0} r_{\A,h}(n)z^n, \quad \frac{\left(\sum_{a\in \A} z^a\right)^h}{1-z} = \sum_{n\geq 0} s_{\A,h}(n)z^n. \label{forms}
 \end{equation}
 When $h=1$, we denote $r_{\A,1}(n)$, $s_{\A,1}(n)$ by $\mathbbm{1}_{\A}(n)$ and $A(n)$ resp..
 
 An $h$-basis is said to be \emph{economical}\footnote{There appears to be conflicting descriptions of \emph{economical} and \emph{thin} $h$-bases in the literature. We believe that the most satisfactory definitions are the following: an $h$-basis $\A\subseteq \N$ is \emph{thin} when $A(x) = \Theta(x^{1/h})$ (cf. Nathanson \cite{nat12}), and \emph{economical} when $r_{\A,h}(n)\ll_\eps n^{\eps}$ (cf. Chapter III, p. 111 of Halberstam \& Roth \cite{halberstam83}). Note that not all economical bases are thin, and, a priori, thin bases are not necessarily economical. In spite of that, economical $h$-basis always satisfy $A(x) = x^{1/h +o(1)}$.} when $r_{\A,h}(n) \ll_{\eps} n^{\eps}$, that is, $r_{\A,h}(n) = O(n^{\eps})$ for every $\eps >0$. The study of such objects dates back to S. Sidon, who in the 1930s inquired about the existence of economical $2$-basis (cf. Erd\H{o}s \cite{erdo56}). A positive answer was given by Erd\H{o}s in the 1950s by means of probabilistic (therefore non-constructive) methods, showing the existence of $\A\subseteq\N$ with $r_{\A,2}(n) = \Theta(\log(n))$. A constructive proof was later found in the 1990s (cf. Kolountzakis \cite{kol95}). Also in the 1990s, Erd\H{o}s and Tetali \cite{erdtet90} settled the $h\geq 3$ case, showing for every such $h$ the existence of an $h$-basis $\A\subseteq\N$ with $r_{\A,h}(n) = \Theta(\log(n))$. The strategy behind their proof can be outlined in three main steps:
 
 \begin{enumerate}[I.]
  \item Define a random sequence $\omega\subseteq\N$ by
  \[ \Pr(n\in \omega) = \frac{f(n)}{n+1}, \]
  where $f(x) := Mx^{1/h}\log(x)^{1/h}$ for some large constant $M>1$; \label{stepI}
  
  \item Show that the expected value of $r_{\omega,h}(n)$ is $\Theta(f(n)^h/n)$; \label{stepII}
  
  \item \label{stepIII} Show that $r_{\omega,h}(n)$ concentrates around its mean. That is,
  \[ \Pr\big(\exists c_1,c_2>0 : \forall n\in\N,~ c_1\mathbb{E}(r_{\omega,h}(n)) \leq r_{\omega,h}(n) \leq c_2\mathbb{E}(r_{\omega,h}(n)) \big) = 1. \] 
 \end{enumerate}
 
 The third step is by far the most involved. The case $h=2$ is rather special, for $r_{\omega,2}(n)$ can be written as a sum of independent Bernoulli trials, hence the third step may be tackled by using Chernoff-type bounds.\footnote{As in Section 8.6, p. 139 of Alon \& Spencer \cite{alon16}.} This is not the case for $h\geq 3$. The way Erd\H{o}s and Tetali worked around this limitation was by showing that the events counted by $r_{\omega,h}(n)$ have low correlation. Vu's alternative proof in Section 5 of \cite{vvu00mc} uses the same principle, albeit with a much more powerful two-sided concentration inequality. In this paper we take each of these aforementioned steps and extend their reach, being relatively faithful to the methods from Erd\H{o}s \& Tetali \cite{erdtet90}. Our studies culminate in the following:
 
 \theoremstyle{plain}
 \newtheorem*{mainthm}{Main Theorem}
 \begin{mainthm}
  Let $h\geq 2$ be fixed and $x_0\in\R_{+}$ arbitrarily large. For any locally integrable, positive real function $f:[x_0,+\infty)\to \R_{+}$ such that
  \begin{enumerate}[(i)]
   \item $\displaystyle{\int_{x_0}^x \frac{f(t)}{t}\mathrm{d}t = \Theta(f(x)) }$ \label{cond1},
   \item \corr{$\displaystyle{x^{1/h}\log(x)^{1/h} \ll f(x) \ll x^{1/(h-1)-\eps}}~$} for some $\eps > 0$; $\phantom{\displaystyle{\int_{x_0}^x}}$ \label{cond2}
  \end{enumerate}
  there is an $h$-basis $\A$ with $A(x) = \Theta(f(x))$ such that
  \[ r_{\A,h+\ell}(n)= \Theta\left(\frac{f(n)^{h+\ell}}{n}\right), \text{ for every } \ell\geq 0. \]
  Furthermore, when $h=2$ \corr{the same conclusion holds under $x^{1/2}\log(x)^{1/2} \ll f(x) \ll x$.}
 \end{mainthm}

 \corr{The lower bound in condition (ii) is the natural threshold for the Erdős--Tetali concentration argument.  The quantity $x^{1/(h-1)}$ is likewise a natural upper barrier for the maxdisfam estimates developed below.  The argument in the published version claimed to reach this barrier up to a logarithmic factor, but a gap in Lemma 5.5 leaves the present proof valid only a fixed power below it when $h\geq3$.  The wider range can be recovered by combining \cite[Theorem 1.4]{taf25} with Theorem \ref{p314} below.} Condition \eqref{cond1} relates to Steps \ref{stepI} and \ref{stepII}. In Appendix \ref{appA} we show that when $f(x) = A(x)$ for some $\A\subseteq\N$, this condition is satisfied if and only if $\A$ is an \emph{$O$-regular plus} (OR$_{+}$) \emph{sequence}, which are sequences such that $A(2x) = O(A(x))$ and $a_{2n} = O(a_{n})$. This appendix is used to explain the idea behind the regularity conditions we assume on sequences in order to achieve our result, with a view towards the general theory of sequences and additive bases, and some applications are given. In Section \ref{sec3} we present a natural way of using an OR$_{+}$ sequence $\A$ to induce a probability measure in the space of sequences so that the expected value of $r_{\omega,h}(n)$ is $\Theta(A(n)^{h}/n)$ for all $h\geq 1$. We will argue that these sequences are the only kind for which one can draw such conclusion.

 \begin{ntt}
  Our use of the asymptotic notations $\Theta,\asymp,O,\ll,o,\sim$ is standard, as well as the ``floor'' and ``ceiling'' functions $\lfloor\cdot\rfloor$, $\lceil \cdot\rceil$. Unless otherwise specified, asymptotic symbols are used assuming $n\to +\infty$ through $\N$ and $x\to +\infty$ through $\R$. When we say ``$P(n)$ \emph{holds for all large} $n$'' we mean that there is $n_0\in\N$ such that $P(n)$ holds for all $n\geq n_0$. Whenever we define a positive real function $f:[x_0,+\infty)\to \R_{+}$, the $x_0\in\R_+$ is to be assumed arbitrarily large; this is just to accommodate functions such as $\log\log\log(x)/\log\log(x)$ without worrying about small $x$.
  
  We use $\Pr$ for the probability measure, $\mathbb{E}$ for expectation and $\Var$ for variance. Given a probability space $(\Omega, \mathcal{F}, \Pr)$, we denote the indicator function of an event $E\in \mathcal{F}$ by $\mathcal{I}_E$. A random variable (abbreviated r.v.) $X:\Omega\to \R$ is said to be a \emph{Bernoulli trial} when $X(\Omega) \subseteq \{0,1\}$. The conditional expectation of a r.v. $X$ given an event $E\in\mathcal{F}$ is denoted by $\mathbb{E}(X|E) := \Pr(E)^{-1}\cdot \int_E X~\mathrm{d}\Pr$.
  
  Finally, whenever we write an asymptotic sign with a superscripted ``a.s.'' we mean that the limit in the definition of that sign holds \emph{almost surely}, i.e. for a subset of $\Omega$ with complement having measure $0$. Please refer to the beginning of Appendix \ref{appA} for the notation regarding \emph{sequences}.
 \end{ntt}

\section{Preliminaries}\label{sec1}
 This section gathers basic preliminary concepts and results that will be used in our arguments; for this reason, everything is unnumbered. The reader may skim through it first and come back as necessary.

\subsection{Probabilistic tools}
 Consider a probability space $(\Omega, \mathcal{F}, \Pr)$. Given a collection $\mathcal{X}$ of r.v.s, denote by $\sigma(\mathcal{X})$ the \emph{$\sigma$-algebra generated by $\mathcal{X}$}, which is the smallest $\sigma$-algebra contained in $\mathcal{F}$ such that all r.v.s $X\in\mathcal{X}$ are measurable. Given a sequence $(X_n)_{n\geq 1}$ of r.v.s, consider an event 
 \[ E \in \bigcap_{N\geq 1} \sigma\left((X_n)_{n\geq N} \right). \]
 When the sequence consists of mutually independent r.v.s, this event is independent of any finite subsequence and is called a \emph{tail event} for the sequence. Such events have either probability $0$ or $1$.\footnote{Kolmogorov's zero--one law (Theorem 2 in Section IV.6, p. 124 of Feller \cite{feller2}).} These observations will be relevant for our discussion in Section \ref{sec2}. In what follows, we present the three essential lemmas from probability theory that shall be used in our considerations.
  
 \theoremstyle{plain}
 \newtheorem*{disjlem}{Disjointness lemma}
 \begin{disjlem}[Lemma 8.4.1, p. 135 of Alon \& Spencer \cite{alon16}]
  Let $\mathcal{E} = \{E_i : i \in I\}$ be a (not necessarily finite) family of events and define $S:= \sum_{E\in\mathcal{E}} \mathcal{I}_{E}$. If $\mathbb{E}(S)<+\infty$, then for all $k\geq 1$ the following holds:
  \[ \Pr\left(\mathcal{D}\subseteq\mathcal{E} \text{ disfam}\footnotemark : |\mathcal{D}| = k \right) \leq \sum_{\substack{\mathcal{J}:\mathcal{J}\subseteq \mathcal{E} \\  \text{disfam, } |\mathcal{J}| = k}} \Pr\left(\bigwedge_{E\in\mathcal{J}} E \right) \leq \frac{\mathbb{E}(S)^k}{k!}. \]
 \end{disjlem}

 \theoremstyle{plain}
 \newtheorem*{chern}{Chernoff bounds}
 \begin{chern}[Theorem 1.8, p. 11 of Tao \& Vu \cite{taovu06}]
  Let $x_1,\ldots,x_n$ be mutually independent r.v.s with $|x_i|\leq \lambda,~\forall i\leq n$ for some fixed constant $\lambda>0$. Taking $S:= x_1+\cdots+x_n$, if $\mathbb{E}(S)>0$ then, for every $\eps >0$,
  \[ \Pr\left(\left|\frac{S}{\mathbb{E}(S)}-1\right|\geq \eps \right)\leq 2e^{-\min\{\eps/2, (\eps/2)^2\} \cdot \mathbb{E}(S)/\lambda }. \]
 \end{chern}
  
 \theoremstyle{plain}
 \newtheorem*{corlem}{Correlation inequality}
 \begin{corlem}[Boppana \& Spencer \cite{bopspe89}]
  Suppose $\Omega = \{\omega_1,\omega_2\ldots\}$ is finite and let $\mathcal{R}\subseteq \Omega$ be a random subset given by $\Pr(\omega_k\in \mathcal{R}) = p_k$, these events being mutually independent. Let $\mathfrak{s}_1,\ldots,\mathfrak{s}_n\subseteq \Omega$ be different subsets of $\Omega$ and, respectively, let $E_1,\ldots,E_n$ denote the events ``$\mathfrak{s}_i\subseteq \mathcal{R}$''. Furthermore, assume $\Pr(E_i)\leq 1/2$ for $1\leq i\leq n$. Then:
  \[ \prod_{i=1}^n \Pr(\bar{E}_i) \leq \Pr\left(\bigwedge_{i=1}^n \bar{E}_i\right) \leq \left(\prod_{i=1}^n \Pr(\bar{E}_i)\right) e^{2\Delta}, \]
  where
  \[ \Delta = \sum_{\substack{1 \leq i < j \leq n \\ \corr{\mathfrak{s}_i\,\cap\, \mathfrak{s}_j \neq \varnothing}}}\Pr\left(E_i\wedge E_j\right) \]
  and $\bar{E}_i$ is the complement of $E_i$. \footnotetext{A \emph{disjoint family} (abbreviated \emph{disfam}) is a mutually independent subfamily of $\mathcal{E}$.}
 \end{corlem}
 
 It is worth noting that, in Section \ref{sec4}, almost all of the techniques used may be broken down into several instances of the disjointness lemma followed by an application of the Borel--Cantelli lemma.\footnote{cf. Section VIII.3, p. 200 of Feller \cite{feller1}.}
  
\subsection{Exact representation functions}
 Fix an integer $h\geq 2$. Let $n\in\N$ and consider an arbitrary sequence $\A\subseteq\N$. Denote by
 \begin{equation*}
  \mathrm{R}_h(n; \A) := \left\{(k_1,\ldots, k_h) \in \A^h : k_1+\cdots + k_h = n\right\}
 \end{equation*}
 the set of $h$-representations of $n$ in $\A$, writing just $\mathrm{R}_h(n)$ when $\A=\N$. We say an $h$-representation $\mathfrak{R}=(k_1,\ldots,k_h)\in \mathrm{R}_h(n; \A)$ is \emph{exact} when $k_i\neq k_j$ whenever $i\neq j$. Denote the set of exact $h$-representations by
 \begin{equation*}
  \mathrm{ER}_h(n; \A) := \bigslant{\left\{ \mathfrak{R}\in\mathrm{R}_h(n;\A): \mathfrak{R} \text{ is exact} \right\}}{\mathfrak{S}_h},
 \end{equation*}
 where $\mathfrak{S}_h$ is the symmetric group in $h$ symbols; that is, we consider two exact representations to be equal when one can be obtained by a permutation of the other, hence $\mathfrak{R}\in \mathrm{ER}_h(n; \A)$ can be thought of as a set $\{k_1,\ldots,k_h\}$ rather than as an $h$-tuple. Similarly, write simply $\mathrm{ER}_h(n)$ when $\A=\N$. Two $h$-representations $\mathfrak{R}, \mathfrak{R}' \in \mathrm{R}_h(n;\A)$ are \emph{disjoint} when $\mathfrak{R}\cap \mathfrak{R}' = \varnothing$. Thus we define two auxiliary representation functions.
 \begin{itemize}
  \item The \emph{exact representation function}:
  \[ \rho_{\A,h}(n) := \left|\mathrm{ER}_h(n;\A)\right|; \]
 
  \item The \emph{maxdisfam} \emph{exact representation function}:
  \[ \rhohat_{\A,h}(n) := \max\left\{|\mathcal{C}|: \mathcal{C}\subseteq \mathrm{ER}_h(n;\A) \substack{\text{ is a maximal disjoint family} \\ \text{ of exact $h$-representations.}}\right\}. \]
 \end{itemize}
 
 While $r_{\A,h}$ is more convenient for induction purposes, $\rho_{\A,h}$ is easier to work with in probabilistic settings. This is because on the one hand, one has equations \eqref{p1}, while on the other, $\rho_{\A,h}$ may be written as
 \begin{equation*}
  \rho_{\A,h}(n) = \sum_{\substack{k_1+\ldots +k_h=n \\ k_1< \cdots < k_h}} \mathbbm{1}_{\A}(k_1)\cdots\mathbbm{1}_{\A}(k_h).
 \end{equation*}
 The $\rhohat_{\A,h}$ function is more abstract. It is useful in probabilistic settings when accompanied by the disjointness lemma. If $\mathcal{C}_n\subseteq\mathrm{ER}_{h}(n;\A)$ is a maximal disjoint family of exact $h$-representations, then
 \begin{equation*}
  \rho_{\A,h}(n) \leq \sum_{k\in\bigcup\mathcal{C}_n} \rho_{\A,h-1}(n-k) \leq h\cdot\rhohat_{\A,h}(n)\cdot \left(\max_{k\in\bigcup\mathcal{C}_n} \rho_{\A,h-1}(n-k) \right), 
 \end{equation*}
 for $|\bigcup \mathcal{C}_n| \leq h\cdot\rhohat_{\A,h}(n)$. This allows us to derive some rough upper bounds for $\rho_{\A,h}$ by induction, but is a fairly limited method since $\rhohat_{\A,h}(n) \leq n/h$.
 
 \begin{xrem}
  The reader familiar with hypergraphs may interpret these functions as counting hyperedges in $\mathcal{H}_{\A,h}^{(n)} = (\mathcal{V},\mathcal{E})$, where $\mathcal{V}:=\A\cap \{0,1,\ldots,n\}$ and $\mathcal{E}:= \mathrm{ER}_h(n;\A)$. Thus $\rhohat_{\A,h}(n)$ is the cardinality of a maximum matching in $\mathcal{H}_{\A,h}^{(n)}$, which when used with the maximum degree $\Delta_{1}(\mathcal{H}) := \max_{u\in\mathcal{V}} \left| \{e \in\mathcal{E} : u \in e \} \right|$ provides some rough upper bounds for $|\mathcal{E}|$. This is a pictorial analogy to keep in mind, but we will not need to make explicit use of it in this paper. For a closely related problem which requires extensive use of this graph-theoretical language, refer to Warnke \cite{war17}.
 \end{xrem}
 
 Another very useful restriction of the set of representations concerns \emph{lower bounded representations}. Given $x_0 \in \R_{+}$, consider the partitions
 \begin{equation*}
  \begin{split}
   \restr{\mathrm{R}_{h}(n;\A)}{\geq x_0} &:= \{\mathfrak{R}\in\mathrm{R}_{h}(n;\A) : \forall k\in\mathfrak{R}, k\geq x_0\}, \\
   \mathrm{R}_{h}(n;\A)\flat_{x_0} &:= \{\mathfrak{R}\in\mathrm{R}_{h}(n;\A) : \exists k\in\mathfrak{R}, k< x_0\}.
  \end{split}
 \end{equation*}
 This is distinguishing those $h$-representations for which the smallest term is at least as large as $x_0$ from those for which some of its elements are smaller than it. In view of what has been discussed, $\restr{\mathrm{ER}_{h}(n;\A)}{\geq x_0}$, $\mathrm{ER}_{h}(n;\A)\flat_{x_0}$, $\restr{\rho_{\A,h}(n)}{\geq x_0}$, $\rho_{\A,h}(n)\flat_{x_0}$, $\restr{\rhohat_{\A,h}(n)}{\geq x_0}$ and $\rhohat_{\A,h}(n)\flat_{x_0}$ are defined accordingly.
  
\subsection{Regular variation}
 Fundamental to the calculations done in this paper are the notions of $O$-regular variation (Subsection \ref{appA2}) and Matuszewska indices (\eqref{dmatusz} in Subsection \ref{appA3}). These are introduced in Appendix \ref{appA}, with a view towards the theory of \emph{sequences} (in the sense of Halberstam \& Roth \cite{halberstam83}) and the general study of additive bases. The concepts necessary for the proof of the Main Theorem are the definitions of OR, PI and OR$_+$ sequences (Propositions \ref{chrreg}, \ref{chrpi} and \ref{chrorpi}, resp.), as well as what we called \emph{the OR--PI lemma} (Lemma \ref{pilem}).
 
\section{The space of sequences (Step \ref{stepI})}\label{sec2}
 Denote by $\mathcal{S}:=\{\A \subseteq \N : \A \text{ is infinite}\}$ the \emph{space of integer sequences}. Say a real sequence $(\alpha_n)_{n\geq 0}$ is a \emph{sequence of probabilities} when $0\leq \alpha_n \leq 1$ for all $n\in \N$, and say it is \emph{proper} when $\sum_{n\geq 0}\alpha_n$ diverges. A proper sequence of probabilities induces a probability measure on $\mathcal{S}$; we denote by $\langle\mathcal{S}, (\alpha_n)_{n}\rangle$ the space $(\mathcal{S},\mathcal{F},\Pr)$ that satisfies the following:
 \begin{enumerate}[(i)]
  \item The events $E_n:=\{\omega: n\in\omega\}$\footnote{Here and throughout this paper we denote by $\omega$ ($\omega \in \mathcal{S}$) a random sequence following the probability distribution induced by $(\alpha_n)_{n\geq 0}$.} are measurable and $\Pr(E_n) = \alpha_n$;

  \item $\{E_0, E_1, E_2,\ldots\}$ is a collection of mutually independent events;

  \item $\mathcal{F}$ is the $\sigma$-algebra induced by the collection $\{E_0, E_1, E_2, \ldots\}$.
 \end{enumerate}
  
 The requirement for $(\alpha_n)_{n\geq 0}$ to be proper is a necessary and sufficient condition to ensure that this is a well-defined probability space according to the Borel--Cantelli lemmas. Indeed, once infinite Cartesian products of probability spaces are established,\footnote{This may be found in several sources, most notably in Halmos \cite{halmos74} (cf. Section 38).} consider a sequence of independent Bernoulli trials with probability $\alpha_n$ of success, then remove those events in which only a finite number of successes occur and identify it with $\sigma((\mathbbm{1}_\omega(n))_{n\geq 0})$.
 
 \begin{xrem}
  As a side note, the reason we chose to work with $\mathcal{S}$ instead of the entire power set of $\N$ is that we will only deal with infinite sequences. It also has the nice one-to-one correspondence:
  \begin{align*}
   \mathcal{S} &\longrightarrow \R/\Z \\
   \A &\longmapsto \textstyle{\left[\sum_{a\in\A} 2^{-(a+1)}\right]} 
  \end{align*}
  thus probability measures on $\mathcal{S}$ correspond naturally to certain measures on the unit circle. The only $\langle\mathcal{S}, (\alpha_n)_{n}\rangle$ corresponding to a translation-invariant measure in the circle is the one with $\alpha_n = 1/2$ for all $n\in \N$, which coincides with the Borel measure on $\R/\Z$.
 \end{xrem}

 Instances of this construction appear in numerous works, most notably in Erd\H{o}s \& R\'enyi \cite{erdren60}. For a classical treatment on related combinatorial and number-theoretical applications, see Chapter III of Halberstam \& Roth \cite{halberstam83}. Essentially, the idea behind these probability measures is to study sequences with certain prescribed rates of growth. Indeed, it is a consequence of Kolmogorov's three-series theorem (cf. Section IX.9 of Feller \cite{feller2}) that if $(X_n)_{n\geq 0}$ is a sequence of independent r.v.s with $\mathbb{E}(X_n)=0$ and $\sum_{n\geq 0} \Var(X_n) < +\infty$, then $\sum_{n\geq 0} X_n$ converges almost surely. The idea of prescribing rates of growth has then the precise meaning implied by the following variant of the strong law of large numbers.\footnote{This is a well-known variant of the strong law, as it is mentioned in Erd\H{o}s \& R\'enyi \cite{erdren60}. A similar statement is present in Chapter III, \S 11 of Halberstam \& Roth \cite{halberstam83}, but with some unnecessary additional hypotheses. Given the fundamental nature of this result in relation to $\langle\mathcal{S}, (\alpha_n)_n \rangle$, we offer it a short proof just for the sake of completeness.}
  
 \begin{thm}[Strong law]\label{slln}
  If $(\alpha_n)_{n\geq 0}$ is a proper sequence of probabilities, then in $\langle\mathcal{S}, (\alpha_n)_n \rangle$ the following holds:
  \[ W(x) \uptxt{a.s.}{\sim} \mathbb{E}(W(x)), \]
  where $W(x) := |\omega\cap [0,x]|$.
 \end{thm}
 \begin{proof}
  First, notice that $\mathbb{E}(W(x)) = \sum_{n\leq x} \alpha_n$. Letting $n_0$ be the smallest number for which $\alpha_{n_0}>0$, define the r.v.s
  \[ X_n := \frac{\mathbbm{1}_{\omega}(n)-\alpha_n}{\sum_{k\leq n} \alpha_k}, \]
  for all $n\geq n_0$, with $X_n \equiv 0$ for $n<n_0$. Since these are just linear transformations being applied to each $\mathbbm{1}_{\omega}(n)$, the r.v.s $X_n$ are still mutually independent. By routine calculations, it follows that $\mathbb{E}(X_n)=0$ and $\Var(X_n) = (\alpha_n - {\alpha_n}^2 )/(\sum_{k\leq n} \alpha_k)^2$. Note that
  \[ \frac{\alpha_n - \alpha_n^2}{\left(\sum_{k\leq n} \alpha_k \right)^2} \leq \frac{1}{\sum_{k< n} \alpha_k } - \frac{1}{\sum_{k\leq n} \alpha_k}, \]
  hence we may telescope the sum $\sum_{n\geq n_0} \Var(X_n)$ and thus it converges.
  
  By the aforementioned theorem of Kolmogorov, it follows that $\sum_{n\geq 0} X_n$ converges almost surely to, say, \corr{$C\in \R$.} Then, for each $n\geq -1$, define $\eps_{n} := C - \sum_{k\leq n} X_k$. Note that $X_k = \eps_{k-1} - \eps_{k}$, thus by partial summation we get
  \begin{align*}
   W(x) - \sum_{n\leq x} \alpha_{n} &= \sum_{n\leq x} X_n \mathbb{E}(W(n)) \\
   &= \sum_{n\leq x} \eps_{n}\alpha_{n+1} + C\cdot \mathbb{E}(W(0)) - \eps_{\lfloor x\rfloor}\cdot\mathbb{E}(W(x)).
  \end{align*}
  Since $\eps_n \uptxt{a.s.}{=} o(1)$, the last term in the RHS is a.s. $o(\mathbb{E}(W(x)))$. Recall that $\sum_{n\leq x} \alpha_{n}$ diverges, thus the first term is also a.s. $o(\mathbb{E}(W(x)))$, hence the theorem is proved.
 \end{proof}
  
 From the many types of tail events in $\langle \mathcal{S}, (\alpha_n)_n\rangle$ one then might study, we focus on the behavior of the r.v.s $r_{\omega,h}(n)$ and $s_{\omega,h}(n)$.

\section{The fundamental lemma (Step \ref{stepII})}\label{sec3}
 Recall that our goal in this step is to describe $\mathbb{E}(r_{\omega,h}(n))$ in a simple way. If we want to have $|\omega\cap[0,x]| \uptxt{a.s.}{=} \Theta(f(x))$ in $\langle\mathcal{S}, (\alpha_n)_{n}\rangle$ for some $f$, by the strong law it is necessary and sufficient that $\sum_{n\leq x} \alpha_n = \Theta(f(x))$. If we assume $f$ to be differentiable, one could then choose $\alpha_n$ to be something like $f'(n)$, so that everything is expressed in terms of $f$. Alternatively, in view of the OR--PI lemma (Lemma \ref{pilem}), we could just assume $f$ to be $O$-regularly varying and having positive increase, so that $\alpha_n$ can be chosen to be something like $f(n)/n$. This is actually preferable, for it can be formulated in a very clean way:

 \begin{defx}[The space $\mathcal{S}_{\A}$]
  When $\A$ is an OR$_{+}$ sequence, let $\mathcal{S}_{\A}$ denote the space $\langle\mathcal{S}, (\alpha_n)_{n}\rangle$ with
  \[ \alpha_{n} := \frac{A(n)}{n+1}. \]
 \end{defx}

 This definition only makes sense when $\A$ is OR$_{+}$ (Proposition \ref{chrorpi}),\footnote{Hence whenever we write $\mathcal{S}_{\A}$ in this paper, $\A$ is being assumed OR$_{+}$.} for it is the only case in which this construction yields $|\omega \cap [0,x]| \uptxt{a.s.}{\asymp} A(x)$. Not only that, but since $\A$ is, in particular, OR (Proposition \ref{chrreg}), it follows that $s_{\omega,h}(x) \uptxt{a.s.}{\asymp} A(x)^h$ for all $h\geq 1$. If one has a suitable function $f$ and still wants $\Theta(f(x))$ as the prescribed rate of growth, it is then sufficient to obtain a sequence $\A\subseteq\N$ with $A(x) = \Theta(f(x))$. Moreover, in this setting, the expected result holds. We first need the following lemma.
 
 \begin{lem}\label{lko}
  Let $\A\subseteq \N$, $h\geq 1$, $n\in\N$ and $k \geq 0$. Then
  \[ r_{\A,h}(n) \leq r_{\A\cup\{k\}, h}(n) \leq r_{\A,h}(n)+\binom{h}{\lfloor h/2\rfloor}\sum_{\ell=1}^{h-1} r_{\A,h-\ell}(n-\ell k) + \delta_{n/h}(k), \]
  where $\delta_{n/h}(k) = 1$ if $k=n/h$ and $0$ otherwise.
 \end{lem}
 \begin{proof}
  The LHS inequality is obvious. For the other side we have:
  \begin{align*}
   r_{\A\cup \{k\},h}(n) &= \sum_{k_1+k_2+\ldots+k_h = n} \mathbbm{1}_{\A\cup \{k\}}(k_1)\mathbbm{1}_{\A\cup \{k\}}(k_2)\ldots \mathbbm{1}_{\A\cup \{k\}}(k_h) \\
   &= \sum_{k_1+k_2+\ldots+k_h = n} \mathbbm{1}_{\A}(k_1)\mathbbm{1}_{\A}(k_2)\ldots \mathbbm{1}_{\A}(k_h) + \sum_{\substack{k_1+k_2+\ldots+k_h = n \\ k \in \{k_1,\cdots,k_h\} \\ k_j\neq k \iff k_j \in \A}} 1.
  \end{align*}
  Thus, assuming $k\notin \A$,
  \[ r_{\A\cup \{k\},h}(n) = r_{\A,h}(n) + \sum_{\ell =1}^{h-1} \binom{h}{\ell}r_{\A,h-\ell}(n-\ell k) + \delta_{n/h}(k), \]
  and the RHS inequality follows.
 \end{proof}
  
 \begin{lem}[Fundamental lemma, $r$ form]\label{t37}
  In $\mathcal{S}_{\A}$ the following holds:
  \begin{equation*}
   \mathbb{E}(r_{\omega,h}(n)) = \Theta\left(\frac{A(n)^h}{n}\right),\quad \forall h\geq 1.
  \end{equation*}
 \end{lem}
 \begin{proof}
  Recall the recursive formulas in \eqref{p1}. We have
  \begin{equation*}
   \mathbb{E}(r_{\omega,h}(n)) = \sum_{k\leq n} \mathbb{E}(r_{\omega,h-1}(k)\mathbbm{1}_{\omega}(n-k)).
  \end{equation*}
  We claim that it is sufficient to show that
  \begin{equation}
   \sum_{k\leq n} \mathbb{E}(r_{\omega,h-1}(k)\mathbbm{1}_{\omega}(n-k)) \asymp \sum_{k\leq n} \mathbb{E}(r_{\omega,h-1}(k))\mathbb{E}(\mathbbm{1}_{\omega}(n-k)). \label{suffsh}
  \end{equation}
  Indeed, since the claim of our lemma holds for $h=1$ by definition, one may apply induction to see that
  \begin{equation}
   \begin{split}
    \sum_{k\leq n} \mathbb{E}(&r_{\omega,h-1}(k))\mathbb{E}(\mathbbm{1}_{\omega}(n-k)) \\
    \asymp~ &\sum_{1\leq k\leq n-1} \frac{A(k)^{h-1}}{k} \frac{A(n-k)}{n-k} \\
    =~ &\sum_{1\leq k\leq n/2} \frac{A(k)^{h-1}}{k} \frac{A(n-k)}{n-k} + \sum_{n/2<k\leq n-1} \frac{A(k)^{h-1}}{k} \frac{A(n-k)}{n-k}.
   \end{split} \label{keyobs}
  \end{equation}
  Then, using that $\A$ is OR$_{+}$, for the ``$\gg$'' side we have
  \begin{align*}
   \sum_{1\leq k\leq n-1} \frac{A(k)^{h-1}}{k} \frac{A(n-k)}{n-k} &\gg \sum_{n/2<k\leq n-1} \frac{A(k)^{h-1}}{k} \frac{A(n-k)}{n-k} \\
   &\gg \frac{A(n)^{h-1}}{n} \sum_{n/2<k\leq n-1} \frac{A(n-k)}{n-k} \\
   &= \frac{A(n)^{h-1}}{n} \sum_{1\leq k< n/2} \frac{A(k)}{k},
  \end{align*}
  which by Lemma \ref{pilem} is $\gg A(n)^h/n$. For the ``$\ll$'' side
  \begin{align*}
   \sum_{1\leq k\leq n/2} \frac{A(k)^{h-1}}{k} &\frac{A(n-k)}{n-k} + \sum_{n/2<k\leq n-1} \frac{A(k)^{h-1}}{k} \frac{A(n-k)}{n-k} \\
   \ll~ &\frac{A(n)}{n}\sum_{1\leq k\leq n/2} \frac{A(k)^{h-1}}{k} + \frac{A(n)^{h-1}}{n}\sum_{n/2<k\leq n-1} \frac{A(n-k)}{n-k},
  \end{align*}
  which, again by Lemma \ref{pilem}, is $\ll A(n)^h/n$.

  Hence, we just need to prove \eqref{suffsh}. Note that the case $h=2$ holds, since the $\mathbbm{1}_{\omega}(n)$ are mutually independent over $n$; hence we apply induction for $h\geq 3$. The r.v.s $r_{\omega,h-1}(k)$ and $\mathbbm{1}_\omega(n-k)$ are not mutually independent in general, but since $\mathbbm{1}_\omega(n-k)$ is a Bernoulli trial, we at least have:
  \begin{equation*}
   \mathbb{E}(r_{\omega,h-1}(k)\mathbbm{1}_{\omega}(n-k)) = \mathbb{E}(\mathbbm{1}_{\omega}(n-k))\mathbb{E}(r_{\omega,h-1}(k)\mid \mathbbm{1}_{\omega}(n-k) = 1).
  \end{equation*}
  Applying Lemma \ref{lko} to $\omega$, $h-1$, $k$ and taking $n-k$,
  \begin{align}
   &\phantom{\leq}\,\,\, \mathbb{E}(\mathbbm{1}_{\omega}(n-k))\mathbb{E}(r_{\omega,h-1}(k)) \phantom{\Bigg(} \nonumber \\
   &\leq \mathbb{E}(\mathbbm{1}_{\omega}(n-k))\mathbb{E}(r_{\omega,h-1}(k)\mid \mathbbm{1}_{\omega}(n-k) = 1) \label{desig} \\
   &\begin{aligned}
    \leq\mathbb{E}(\mathbbm{1}_{\omega}(n-k))\Bigg( \mathbb{E}(r_{\omega,h-1}(k)) ~+~ \delta_{k/(h-1)}&(n-k) ~+ \\
    +~\binom{h-1}{\lfloor \frac{h-1}{2}\rfloor} &\sum_{\ell =1}^{h-2} \mathbb{E}(r_{\omega,h-\ell-1}(\ell k-(\ell-1)n))\Bigg).
   \end{aligned} \nonumber
  \end{align}
  The apparent problem in this estimate is the summation on the upper bound. We show that it ends up being negligible. For $h\geq 3$ we have
  \begin{align*}
   &\phantom{=~}
    \sum_{k\leq n} \mathbb{E}(\mathbbm{1}_{\omega}(n-k)) \Bigg(\binom{h-1}{\lfloor \frac{h-1}{2}\rfloor}\sum_{\ell =1}^{h-2} \mathbb{E}(r_{\omega,h-\ell-1}(\ell k-(\ell-1)n)) \delta_{k/(h-1)}(n-k)\Bigg) \\
   &\begin{aligned}
    =\binom{h-1}{\lfloor \frac{h-1}{2}\rfloor}\sum_{\ell =1}^{h-2}\sum_{k\leq n} \mathbb{E}(r_{\omega,h-\ell-1}(\ell k-(\ell -1)n))\mathbb{E}&(\mathbbm{1}_{\omega}(n-k)) + \\ +~ &\sum_{k\leq n} \delta_{(h-1)n/h}(k)\mathbb{E}(\mathbbm{1}_{\omega}(n-k))
   \end{aligned} \\
   &\ll \sum_{\ell =1}^{h-2}\sum_{k\leq n} \mathbb{E}(r_{\omega,h-2}(\ell k-(\ell-1)n))\mathbb{E}(\mathbbm{1}_{\omega}(n-k)) + \frac{A(n)}{n} \\
   &= \sum_{\ell =1}^{h-2}\sum_{k\leq n} \mathbb{E}(r_{\omega,h-2}(n- \ell k))\mathbb{E}(\mathbbm{1}_{\omega}(k)) +\frac{A(n)}{n}.
  \end{align*}
  Since ``$n-\ell k \geq 0$''$\iff$``$k \leq n/\ell$'', this summation can be subdivided as follows:
  \begin{align*}
   \sum_{k\leq n} \mathbb{E}(r_{\omega,h-2}&(n- \ell k))\mathbb{E}(\mathbbm{1}_{\omega}(k)) \\
   \asymp~ &\sum_{1\leq k\leq \frac{n}{2\ell}} \frac{A(n-\ell k)^{h-2}}{n-\ell k}\frac{A(k)}{k} + \sum_{\frac{n}{2\ell} < k\leq \frac{n}{\ell}-1} \frac{A(n-\ell k)^{h-2}}{n-\ell k}\frac{A(k)}{k} \\
   \ll~ & \frac{A(n)^{h-2}}{n}\sum_{1\leq k\leq \frac{n}{2\ell}} \frac{A(k)}{k} + \frac{A(n)}{n}\sum_{\frac{n}{2\ell} < k\leq \frac{n}{\ell}-1} \frac{A(n-\ell k)^{h-2}}{n-\ell k},
  \end{align*}
  therefore, by Lemma \ref{pilem}, this is $O(A(n)^{h-1}/n)$ for all \corr{$\ell\geq 1$}. With this, we conclude from \eqref{desig} that
  \begin{align}
   &\sum_{k\leq n} \mathbb{E}(r_{\omega,h-1}(k))\mathbb{E}(\mathbbm{1}_{\omega}(n-k)) \nonumber \\
   \leq~ &\sum_{k\leq n} \mathbb{E}(r_{\omega,h-1}(k)\mathbbm{1}_{\omega}(n-k)) \label{desig2} \\
   \leq~ &\sum_{k\leq n} \mathbb{E}(r_{\omega,h-1}(k))\mathbb{E}(\mathbbm{1}_{\omega}(n-k)) + O\left(\frac{A(n)^{h-1}}{n}\right), \nonumber
  \end{align}
  hence \eqref{suffsh} follows, and our proof is complete.
 \end{proof}
 
 Two variations of this lemma will be important to us, one for exact ($\rho$) and other for lower bounded exact ($\restr{\rho}{\geq}$) representation functions. When referenced along the text, consider them in conjunction with Lemma \ref{t37}.
 
 \begin{lem}[Fundamental lemma, $\rho$ form]\label{exct37}
  In $\mathcal{S}_{\A}$ the following holds:
  \[ \mathbb{E}(\rho_{\omega,h}(n)) = \frac{1}{h!}\mathbb{E}(r_{\omega,h}(n)) + O\left(\frac{A(n)^{h-1}}{n}\right), \quad\forall h\geq 2. \]
 \end{lem}
 \begin{proof}
  All we need to do is show that
  \[ \mathbb{E}\left(|\{\mathfrak{R}\in\mathrm{R}_{h}(n;\omega):\mathfrak{R}\text{ is non-exact}\}|\right) = O\left(\frac{A(n)^{h-1}}{n}\right). \]
  This is clearly true for $h=2$, thus we may focus on $h\geq 3$. For this case, consider the non-probabilistic estimate:
  \begin{equation*}
  |\{\mathfrak{R}\in\mathrm{R}_{h}(n;\omega):\mathfrak{R}\text{ is non-exact}\}| \leq \binom{h}{2} \sum_{k\leq n/2} r_{\omega,h-2}(n-2k)\mathbbm{1}_{\omega}(k).
  \end{equation*}
  To analyze this summation, let $\mathcal{N}_{\omega,h}(n) := \sum_{k\leq n/2} r_{\omega,h-1}(n-2k)\mathbbm{1}_{\omega}(k)$. We are going to prove that
  \begin{equation}
   \mathbb{E}(\mathcal{N}_{\omega,h}(n)) = \Theta\left(\frac{A(n)^h}{n}\right),\quad \forall h\geq 2. \label{nnex1}
  \end{equation}
  Just as in \eqref{desig} on the proof of Lemma \ref{t37}, we apply Lemma \ref{lko} for $\omega$, $h-1$, $n-2k$ and $k$ to obtain:
  \begin{align*}
    &\phantom{\leq}\,\,\, \mathbb{E}(\mathbbm{1}_{\omega}(k))\mathbb{E}(r_{\omega,h-1}(n-2k)) \phantom{\Bigg(} \\
    &\leq \mathbb{E}(\mathbbm{1}_{\omega}(k))\mathbb{E}(r_{\omega,h-1}(n-2k)\mid \mathbbm{1}_{\omega}(k) = 1) \\
    &\begin{aligned}
     \leq\mathbb{E}(\mathbbm{1}_{\omega}(k))\Bigg( \mathbb{E}(r_{\omega,h-1}(n-2k)) ~+~ \delta_{(n-2k)/(h-1)}(k&) ~+ \\
     +~\binom{h-1}{\lfloor \frac{h-1}{2}\rfloor} &\sum_{\ell =1}^{h-2} \mathbb{E}(r_{\omega,h-\ell-1}(n-(\ell+2)k))\Bigg).
    \end{aligned}
  \end{align*}
  By calculations similar to the ones in Lemma \ref{t37} it follows
  \begin{align*}
   \mathbb{E}(\mathcal{N}_{\omega,h}(n)) &= \sum_{k\leq n/2} \mathbb{E}(r_{\omega,h-1}(n-2k)\mathbbm{1}_{\omega}(k)) \\
   &= \sum_{k\leq n/2} \mathbb{E}(r_{\omega,h-1}(n-2k))\mathbb{E}(\mathbbm{1}_{\omega}(k)) + O\left(\frac{A(n)^{h-1}}{n}\right).
  \end{align*}
  Finally, by using Lemma \ref{t37},
  \begin{align*} \sum_{k\leq n/2} \mathbb{E}(r_{\omega,h-1}&(n- 2k))\mathbb{E}(\mathbbm{1}_{\omega}(k)) \\
  \asymp~ &\sum_{1\leq k\leq \frac{n}{4}} \frac{A(n-2k)^{h-1}}{n-2k}\frac{A(k)}{k} + \sum_{\frac{n}{4} < k\leq \frac{n}{2}-1} \frac{A(n-2k)^{h-1}}{n- 2k}\frac{A(k)}{k},
  \end{align*}
  thus \eqref{nnex1} follows from Lemma \ref{pilem}.
 \end{proof}
 
 \begin{lem}[Fundamental Lemma, $\restr{\rho}{\geq}$ form]\label{rbulk}
  Given $h\geq 2$, for every $\eps>0$ there is $v_h(\eps)>0$ such that in $\mathcal{S}_{\A}$ the following holds:
  \[ \mathbb{E}(\restr{\rho_{\omega,h}(n)}{\geq \eps n}) \geq (1-v_h(\eps))\mathbb{E}(\rho_{\omega,h}(n)) \text{\emph{ for all large $n$.}} \]
  Furthermore, if $\eps = \eps(n) \to 0^{+}$ as $n\to +\infty$, then for every positive $\gamma < \mlow(A)$ one can take $v_h(\eps)$ in a manner that satisfies
  \begin{equation}
   v_h(\eps) \ll \eps^{\gamma/2^{h-2}} \quad (\text{as } n\to +\infty ), \label{clncut}
  \end{equation}
  where $\mlow(A)$ is the lower Matuszewska index (cf. \eqref{dmatusz}) of $A$.
 \end{lem}
 \begin{proof}
  Recall that $\rho_{\omega,h}(n)\flat_{\eps n} = \rho_{\omega,h}(n) - \restr{\rho_{\omega,h}(n)}{\geq \eps n}$ is counting those exact $h$-representations in which at least one term is smaller than $\eps n$. All we need to do is show that $\mathbb{E}(\rho_{\omega,h}(n)\flat_{\eps n}) \leq v_{\corr{h}}(\eps) \mathbb{E}(\rho_{\omega,h}(n))$.
  
  An exact $h$-representation $\mathfrak{R}=\{k_1,\ldots,k_h\}$ of $n$ which has $k_1< \eps n$ must have $k_h > \frac{(1-\eps)}{h-1}n$, therefore
  \begin{align} \mathbb{E}(\rho_{\omega,h}(n)\flat_{\eps n}) &= \mathbb{E}\Bigg(\sum_{\substack{k_1+\ldots +k_h = n \\ k_1 < k_2 < \cdots < k_h \\ k_1 < \eps n}} \mathbbm{1}_{\omega}(k_1)\cdots\mathbbm{1}_{\omega}(k_h) \Bigg) \nonumber \\
  &= \sum_{\substack{k_1+\ldots +k_h = n \\ k_1 < k_2 < \cdots < k_h \\ k_1 < \eps n}} \frac{A(k_1)}{k_1+1}\cdots \frac{A(k_h)}{k_h + 1} \nonumber \\
  &< \frac{h-1}{(1-\eps)}\frac{A(n)}{n + \frac{1}{1-\eps}} \sum_{\substack{k_1+\ldots +k_{h-1} \leq \frac{h-2-\eps}{h-1}n \\ k_1 < k_2 < \cdots < k_{h-1} \\ k_1 < \eps n}} \frac{A(k_1)}{k_1+1}\cdots \frac{A(k_{h-1})}{k_{h-1}+1}, \label{bulk1}
  \end{align}  
  where the last sum can be estimated by
  \begin{align}
   \sum_{\ell \leq \frac{h-2-\eps}{h-1}n} &\sum_{\substack{k_1+\ldots +k_{h-1} = \ell \\ k_1 < k_2 < \cdots < k_{h-1} \\ k_1 < \eps n}} \frac{A(k_1)}{k_1+1}\cdots \frac{A(k_{h-1})}{k_{h-1}+1} \nonumber \\
   &\leq \sum_{\ell \leq \eps^{1/2} n} \mathbb{E}(\rho_{\omega, h-1}(\ell)) + \sum_{\eps^{1/2} n < \ell \leq n} \sum_{\substack{k_1+\ldots +k_{h-1} = \ell \\ k_1 < k_2 < \cdots < k_{h-1} \\ k_1 < \eps^{1/2} \ell}} \frac{A(k_1)}{k_1+1}\cdots \frac{A(k_{h-1})}{k_{h-1}+1} \nonumber \\
   &= \sum_{\ell \leq \eps^{1/2} n} \mathbb{E}(\rho_{\omega, h-1}(\ell)) + \sum_{\eps^{1/2} n < \ell \leq n} \mathbb{E}(\rho_{\omega,h-1}(\ell)\flat_{\eps^{1/2}\ell}). \label{bulk2}
  \end{align}
  By Lemma \ref{exct37}, there are $t_h,T_h>0$ such that, for large $n$,
  \[ t_h \frac{A(n)^{h}}{n} \leq \mathbb{E}(\rho_{\omega,h}(n)) \leq T_h \frac{A(n)^{h}}{n}. \]
  In addition, recall that $\A$ is OR$_{+}$, thus its lower Matuszewska index $\mlow(A)$ is positive. Therefore for every $0 < \gamma < \mlow(A)$ there must be some $M = M_\gamma \in \R_{+}$ for which $A(\eps n) \leq M \eps^{\gamma} A(n)$ for all large $n$. Using that, we can prove our statement by induction.
  
  When $h=2$, by Lemma \ref{pilem}, there exists $C\in\R_{+}$ such that, for large $n$,
  \begin{align*} \mathbb{E}(\rho_{\omega,2}(n)\flat_{\eps n}) &= \sum_{k < \eps n} \frac{A(k)}{k+1}\frac{A(n-k)}{n-k+1} \\
  &\leq \frac{A(n)}{(1-\eps)n+1} \sum_{k < \eps n} \frac{A(k)}{k+1} \leq C\cdot \frac{A(n)}{n} A(\eps n),
  \end{align*}
  thus, since $A(\eps n) \leq M \eps^{\gamma} A(n)$, our statement applies for $h=2$, including the upper bound for $v_2(\eps)$ in \eqref{clncut} for $\eps=\eps(n)\to 0^{+}$. Supposing it holds for $h-1$, from \eqref{bulk2} we get
  \begin{align*}
   \sum_{\ell \leq \eps^{1/2} n} \mathbb{E}(&\rho_{\omega, h-1}(\ell)) + \sum_{\eps^{1/2} n < \ell \leq n} \mathbb{E}(\rho_{\omega,h-1}(\ell)\flat_{\eps^{1/2}\ell}) \\
   &\leq C T_{h-1} M^{h-1}\eps^{(h-1)\gamma/2} A(n)^{h-1} + T_{h-1} v_{h-1}(\eps^{1/2})\sum_{\ell \leq n} \frac{A(\ell)^{h-1}}{\ell} \\
   &\leq C T_{h-1} \left(M^{h-1}\eps^{(h-1)\gamma/2} + v_{h-1}(\eps^{1/2}) \right) A(n)^{h-1},
  \end{align*}
  for large $n$ and some $C\in\R_{+}$ that comes from Lemma \ref{pilem}. Substituting this in \eqref{bulk1} yields
  \begin{align*}
   \mathbb{E}(\rho_{\omega,h}(n)\flat_{\eps n}) < \frac{h-1}{(1-\eps)}\frac{n}{n+\frac{1}{1-\eps}}C T_{h-1} \left(M^{h-1}\eps^{(h-1)\gamma/2} + v_{h-1}(\eps^{1/2}) \right) \frac{A(n)^{h}}{n}.
  \end{align*}
  We have that $n/(n+\frac{1}{1-\eps}) \sim 1$ as $n\to +\infty$ and every other term in this coefficient is either bounded or independent of $\eps$, except for $\eps^{(h-1)\gamma/2}$ and $v_{h-1}(\eps^{1/2})$. Since the induction step includes \eqref{clncut}, we deduce that these two terms must remain bounded, and if $\eps=\eps(n)\to 0^{+}$ then they must vanish like $O(\eps^{\gamma/2^{h-2}})$. The proof is then complete.
 \end{proof}
  
 We finish this section with an important remark about Lemma \ref{t37}.
 
 \begin{rem}\label{imprem}
  Let $g:[x_0,+\infty)\to\R_{+}$ be a positive real function satisfying the conditions of Lemma \ref{pilem}. Based on this lemma, for each $2\leq \ell \leq h$ let $c_\ell$ be some positive constant for which
  \[ \sum_{x_0 < k \leq n-x_0} \frac{g(k)^{\ell-1}}{k}\frac{g(n-k)}{n-k} > c_\ell \frac{g(n)}{n} \]
  holds for all large $n$. In Lemma \ref{t37}, in view of \eqref{desig2}, it is possible to deduce from the equations in \eqref{keyobs} that if $\A$ is such that $A(n) > M g(n)$ for some $M > 0$ and all large $n$, then
  \begin{equation*} \mathbb{E}(r_{\omega,h}(n)) > M^{h} \left(\prod_{\ell = 2}^h c_{\corr{\ell}}\right) \frac{g(n)^h}{n}, \quad \text{$\forall$suff. large $n$}. \end{equation*}
  Note that the same holds when inequalities are reversed. These observations are key for the proof of the Main Theorem.
 \end{rem}

\section{Concentration (Step \ref{stepIII})}\label{sec4}
 Now that the general setup is established, we will start working with the conditions from the Main Theorem. We start with the two lemmas that constitute the most technical parts of the proof.

\subsection{Lemmas}
 The following results have a similar theme concerning the condition $A(x) \ll x^{1/(h-1)}$. As mentioned in Section \ref{sec1}, our methods are naturally limited by the fact that $\rhohat_{\omega,h}(n) \leq n/h$; this is basically why $x^{1/(h-1)}$ pops up, and the purpose of the following lemmas is to make the most out of this restraint. The first lemma says that non-exact $h$-representations are, almost surely, conveniently scarce in $\mathcal{S}_{\A}$ for such $\A$.
  
 \begin{lem}\label{o1rhat}
  Let $h\geq 2$ be fixed. If $A(x)\ll x^{1/(h-1)}$, then in $\mathcal{S}_{\A}$ the following holds:
  \begin{equation*}
   \rho_{\omega,h-1}(n) \uptxt{a.s.}{=} O\left(\frac{\log(n)}{\log\log(n)}\right),
  \end{equation*}
  \begin{equation}
   |\{\mathfrak{R}\in\mathrm{R}_{h}(n;\omega):\mathfrak{R}\text{ is non-exact}\}| \uptxt{a.s.}{=} O\left(\frac{\log(n)}{\log\log(n)}\right). \label{loll2}
  \end{equation}
  Furthermore,
  \begin{equation*}
   \rho_{\omega,h}(n) \uptxt{a.s.}{\ll} \sum_{k\in\bigcup\mathcal{C}_n} \rhohat_{\omega,h-1}(n-k), \label{loll3}
  \end{equation*}
  where $\mathcal{C}_n \subseteq \mathrm{ER}_{h}(n;\omega)$ is a maximal disjoint family.
 \end{lem}
 \begin{proof}
  We start by proving the first two estimates. Just as in Lemma \ref{exct37}, we have the non-probabilistic estimate
  \begin{equation*}
  |\{\mathfrak{R}\in\mathrm{R}_{\ell}(n;\omega):\mathfrak{R}\text{ is non-exact}\}| \leq \binom{\ell}{2} \mathcal{N}_{\omega,\ell -1}(n)
  \end{equation*}
  for $\ell\geq 3$, where $\mathcal{N}_{\omega,\ell}(n) = \sum_{k\leq n/2} r_{\omega,\ell-1}(n-2k)\mathbbm{1}_{\omega}(k)$. Note that this is counting the number of solutions to
  \begin{equation} 2 x_1 + x_2 + \ldots + x_{\ell} = n, \quad x_i\in\omega. \label{eqeps} \end{equation}
 
  Consider now $\widehat{\mathcal{N}}_{\omega,\ell}(n)$ to be the cardinality of the largest maximal collection of disjoint \emph{multisets}\footnote{A \emph{multiset} is a set that allows multiple instances of an element.} $\{k_1,k_2,\ldots,k_\ell\}\subseteq\omega$ satisfying \eqref{eqeps}. We may then deduce the following non-probabilistic inequalities:
  \begin{align}
   r_{\omega,\ell}(n) &\leq \ell! \ell \left(\max_{k\leq n}r_{\omega,\ell-1}(k)\right)\rhohat_{\omega,\ell}(n) + \binom{\ell}{2} \mathcal{N}_{\omega,\ell-1}(n) \label{funnyest} \\
   \mathcal{N}_{\omega,\ell}(n) &\leq \ell! \ell^2 \left(\max_{k\leq n} \{\mathcal{N}_{\omega,\ell-1}(k)\} + \max_{k\leq n} \{r_{\omega,\ell-1}(k)\}\right) \widehat{\mathcal{N}}_{\omega,\ell}(n). \label{funnyest2}
  \end{align}
  Indeed, \eqref{funnyest} follows by noticing that when $\mathcal{C}\subseteq \mathrm{ER}_{\ell}(n;\omega)$ is a maximal collection of pairwise disjoint representations with $|\mathcal{C}|=\rhohat_{\omega,\ell}(n)$, we have $|\bigcup \mathcal{C}| = \ell\cdot \rhohat_{\omega,\ell}(n)$ and $\mathfrak{R}\cap(\bigcup \mathcal{C}) \neq\varnothing$ for every $\mathfrak{R}\in\mathrm{ER}_{\ell}(n;\omega)$; thus the first term bounds the number of exact $\ell$-representations, whereas the latter bounds the number of non-exact ones. The reasoning for \eqref{funnyest2} is analogous.
  
  To prove our lemma, we just need to show that $\mathcal{N}_{\omega,h-1}(n) \uptxt{a.s.}{=} O\left(\frac{\log(n)}{\log\log(n)}\right)$. The estimate in \eqref{funnyest2} allows us to break this into the following three items:
  \begin{enumerate}[(i)]
   \item $r_{\omega,h-2}(n) \uptxt{a.s.}{=} O(1)$; \label{l51i}
   \item $\mathcal{N}_{\omega,h-2}(n) \uptxt{a.s.} = O(1)$; \label{l51ii}
   \item $\widehat{\mathcal{N}}_{\omega,h-1}(n) \uptxt{a.s.} = O(\log(n)/\log\log(n)).$ \label{l51iii}
  \end{enumerate}
  The first two are dealt with by induction. Since $r_{\omega,2}(n)\ll \rhohat_{\omega,2}(n)+1$ and $\mathcal{N}_{\omega,2}(n) \ll \widehat{\mathcal{N}}_{\omega,2}(n)$, to show \eqref{l51i} and \eqref{l51ii} it is sufficient to, in view of \eqref{funnyest} and \eqref{funnyest2}, show that:
  \begin{enumerate}[(i)*]
   \item $\rhohat_{\omega,\ell}(n) \uptxt{a.s.}{=} O(1), \text{ for all } 2\leq \ell \leq h-2$; \label{l51is}
   \item $\widehat{\mathcal{N}}_{\omega,\ell}(n) \uptxt{a.s.} = O(1), \text{ for all } 2\leq \ell \leq h-2$. \label{l51iis}
  \end{enumerate}
  At this point, we have managed to reduce all our problems to estimates involving maxdisfam-type counting functions, thus the next natural step is to use the disjointness lemma! We now prove \eqref{l51is}*, \eqref{l51iis}* and \eqref{l51iii}. \medskip 
  
  \noindent
  $\bullet \textbf{ Items \eqref{l51is}* and \eqref{l51iis}*:}$ Starting with \eqref{l51iis}*, from \eqref{nnex1} in Lemma \ref{exct37} we know that $\mathbb{E}(\mathcal{N}_{\omega,\ell}(n)) = \Theta(A(n)^{\ell}/n)$. Thus, as the function $\widehat{\mathcal{N}}_{\omega,\ell}(n)$ is counting mutually independent events, applying the disjointness lemma yields, in view of having $A(x)\ll x^{1/(h-1)}$,
  \begin{equation*}
   \Pr\left(\widehat{\mathcal{N}}_{\omega,\ell}(n) \geq M \right) \leq \frac{\mathbb{E}(\mathcal{N}_{\omega,\ell}(n))^{M}}{M!} \ll \left(\frac{A(n)^{h-2}}{n}\right)^M \ll n^{-M/(h-1)}.
  \end{equation*}
  Therefore, if we take $M > 2(h-1)$, it follows from the Borel--Cantelli lemma that $\widehat{\mathcal{N}}_{\omega,\ell}(n) \uptxt{a.s.}{=} O(1)$ for all $\ell$ considered. The case \eqref{l51is}* for $\rhohat_{\omega,\ell}(n)$ goes analogously. \medskip
 
  \noindent
  $\bullet \textbf{ Item \eqref{l51iii}:}$ Once again we apply the disjointness lemma, but this time we shall take $M = M_n$ to be a slowly growing function in $n$. First, we have
  \begin{equation}
   \Pr\left(\widehat{\mathcal{N}}_{\omega,h-1}(n) \geq M_n \right) \leq \frac{\mathbb{E}(\mathcal{N}_{\omega,h-1}(n))^{M_n}}{M_n!}. \label{logloglog}
  \end{equation}
  Let $M_n := \lceil 2\log(n)/\log\log(n) \rceil$ for $n > e^e$. Since $M_n$ is unbounded and non-decreasing, we have that for every $\eps>0$ there is $n_0\in\N$ such that
  \[ (1-\eps)M_n\log(M_n) \leq \log(M_n!) \leq (1+\eps) M_n\log(M_n),\quad \forall n\geq n_0. \]
  Thus, keeping in mind that $\mathbb{E}(\mathcal{N}_{\omega,h-1}(n)) < C$ for some $C>0$ and all large $n$, we may apply ``$-\log$'' in both sides of \eqref{logloglog} to obtain, for large $n$,
  \begin{align*}
   -\log\left(\Pr\left({\color{red} \widehat{\mathcal{N}}_{\omega,h-1}(n)} \geq M_n \right)\right) &\geq \log(M_n!) - M_n\log(C) \\
   &\geq (1-\eps)M_n\log(M_n) \\
   &\geq (1-\eps)\frac{2\log(n)}{\log\log(n)}\log\left(\frac{2\log(n)}{\log\log(n)}\right) \\
   &\geq 2(1-\eps)\log(n) - o(\log(n)).
  \end{align*}
  Therefore $\Pr(\widehat{\mathcal{N}}_{\omega,h-1}(n) \geq M_n ) \leq n^{-2(1-\eps)+o(1)}$. Choosing $\eps < 1/2$, we conclude from the Borel--Cantelli lemma that $\widehat{\mathcal{N}}_{\omega,h-1}(n) \uptxt{a.s.}{=} O(M_n)$, thus completing the proof of \eqref{loll2}. As before, the reasoning for $\rho_{\omega,h-1}(n)$ is analogous. \medskip
 
  Finally, given $\ell\geq 2$, consider $\mathcal{C}^{(\ell)}_n \subseteq \mathrm{ER}_{\ell}(n;\omega)$ a maximal disjoint family of exact $\ell$-representations of $n$. By definition, we will have
  \[ \rho_{\omega,\ell}(n) \leq \sum_{k\in\bigcup\mathcal{C}^{(\ell)}_n} \rho_{\omega,\ell-1}(n-k). \]
  From item \eqref{l51i} we know that $\rho_{\omega,\ell-1}(n) \uptxt{a.s.}{=} O(1)$ for $2 \leq \ell \leq h-1$, therefore, since $|\bigcup\mathcal{C}^{(\ell)}_n| = \ell\cdot\rhohat_{\omega,\ell}(n)$, the above implies $\rho_{\omega,\ell}(n) \uptxt{a.s.}{\ll} \rhohat_{\omega,\ell}(n)$ for such $\ell$, thus the last estimate stated in the lemma follows.
 \end{proof}
 
 The second lemma is similar to Lemma 11 from Erd\H{o}s \& Tetali \cite{erdtet90}; the goal is to show that we will have small ``$\Delta$'' in our applications of the correlation inequality. By an abuse of notation, given $\mathfrak{R}\in\mathrm{ER}_h(n)$ we shall also use ``$\mathfrak{R}$'' to denote the event ``$\mathfrak{R}\in\mathrm{ER}_h(n;\omega)$''.
  
 {\color{red} \begin{lem}\label{upll}
  Let $h\geq 2$ be fixed. For every $\eps>0$ there is $M_0\in\R_{+}$ such that, if $A(x)< M^{-1} x^{1/(h-1)}$ for some $M>M_0$ and all sufficiently large $x$, then in $\mathcal{S}_{\A}$ the following holds: 
  \begin{equation*}
   \sum_{\substack{\mathfrak{R},\mathfrak{R}'\in \mathrm{ER}_h(n) \\ 1\leq |\mathfrak{R}\cap\mathfrak{R'}|\leq h-2}} \Pr\left( \mathfrak{R} \wedge \mathfrak{R}'\right) < \eps\, \mathbb{E}(\rho_{\omega,h}(n))
  \end{equation*}
  for all sufficiently large $n$.
 \end{lem} }
 \begin{proof}
  Rewriting this sum as
  \begin{equation*}
   \sum_{\substack{\mathfrak{R},\mathfrak{R}'\in \mathrm{ER}_h(n) \\ 1\leq |\mathfrak{R}\cap\mathfrak{R'}|\leq h-2}} \Pr\left( \mathfrak{R} \wedge \mathfrak{R}'\right) = \sum_{\ell=1}^{h-2} \sum_{\substack{\mathfrak{R},\mathfrak{R}'\in \mathrm{ER}_h(n) \\ |\mathfrak{R}\cap\mathfrak{R'}|=\ell}} \Pr\left( \mathfrak{R} \wedge \mathfrak{R}'\right),
  \end{equation*}
  we just need to estimate each of the terms for $\ell$ fixed. Taking $\mathfrak{R}\cap\mathfrak{R}'$ with $|\mathfrak{R}\cap\mathfrak{R}'|=\ell$, say we have
  \[ \mathfrak{R} = \{x_1,\ldots,x_\ell, y_1,\ldots,y_{h-\ell}\} \text{ and } \mathfrak{R}' = \{x_1,\ldots,x_\ell, z_1,\ldots,z_{h-\ell}\}, \]
  with $\sum_{i=1}^\ell x_i = k$ and $y_i \neq z_j$ for all $i,j$. Letting $\sum^*$ denote the sum over exact representations and abusing the notation again by writing ``$\Pr(x_i)$'' to denote ``$\Pr(\mathbbm{1}_{\omega}(x_i)=1)$'', we have:
  \begin{align*}
   \sum_{\substack{\mathfrak{R},\mathfrak{R}'\in \mathrm{ER}_h(n) \\ |\mathfrak{R}\cap\mathfrak{R'}|=\ell}} &\Pr\left( \mathfrak{R} \wedge \mathfrak{R}'\right) \\
   =~ &\sum_{k\leq n} ~\sideset{}{^{*}}\sum_{\substack{x_1+\ldots+x_\ell = k \\ y_1+\ldots+y_{h-\ell} = n-k \\ z_1+\ldots+z_{h-\ell} = n-k}} \Pr(x_1)\cdots\Pr(x_\ell)\Pr(y_1)\cdots\Pr(y_{h-\ell})\Pr(z_1)\cdots\Pr(z_{h-\ell}) \\
   \leq~ &\sum_{k\leq n} \left(~\sideset{}{^{*}}\sum_{x_1+\ldots +x_\ell = k} \Pr(x_1)\cdots\Pr(x_\ell)\right) \left(~\sideset{}{^{*}}\sum_{y_1+\ldots +y_{h-\ell} = n-k} \Pr(y_1)\cdots\Pr(y_{h-\ell})\right)^2 \\
   =~ &\sum_{k\leq n} \mathbb{E}(\rho_{\omega,\ell}(k))\mathbb{E}(\rho_{\omega,h-\ell}(n-k))^2.
  \end{align*}
  
  {\color{red} Since $A(m)<M^{-1}m^{1/(h-1)}$ for all sufficiently large $m$, Lemmas \ref{t37} and \ref{exct37} imply that $\mathbb{E}(\rho_{\omega,h-\ell}(m)) \ll M^{-(h-\ell)}m^{(1-\ell)/(h-1)} \ll M^{-(h-\ell)}$ for all large $m$ (since $1\leq\ell\leq h-2$), where by Remark \ref{imprem} the implied constant is independent of $\A$. Therefore, for some sufficiently large fixed $N$,
  \begin{align*}
   \sum_{k\leq n-N} \mathbb{E}(\rho_{\omega,\ell}(k)) \mathbb{E}(\rho_{\omega,h-\ell}(n-k))^2 &\ll M^{-(h-\ell)} \sum_{k\leq n} \mathbb{E}(\rho_{\omega,\ell}(k)) \mathbb{E}(\rho_{\omega,h-\ell}(n-k)) \\
   &\leq M^{-(h-\ell)}\mathbb{E}(r_{\omega,h}(n)) \ll M^{-(h-\ell)}\mathbb{E}(\rho_{\omega,h}(n)),
  \end{align*}
  where the middle inequality follows by expanding the expectations and using $\Pr(\mathfrak{R})\Pr(\mathfrak{R}') \leq \Pr(\mathfrak{R}\wedge\mathfrak{R}')$. The remaining $N$ terms are $O(A(n)^{\ell}/n) = o(A(n)^h/n)$. Therefore,
  \[ \sum_{\substack{\mathfrak{R},\mathfrak{R}'\in \mathrm{ER}_h(n) \\ |\mathfrak{R}\cap\mathfrak{R}'|=\ell}} \Pr(\mathfrak{R}\wedge\mathfrak{R}') \ll (M^{-(h-\ell)}+o(1))\, \mathbb{E}(\rho_{\omega,h}(n)). \]
  Choosing $M_0$ and $n$ sufficiently large proves the lemma.}
 \end{proof}
 
 We now prove the Main Theorem.
 
\subsection{Proof of the Main Theorem}
 Fix $h\geq 2$ an integer. Take a positive and locally integrable real function $f:[x_0,+\infty)\to\R_{+}$ \corr{satisfying the hypotheses of the Main Theorem. In particular,}
 \[ x^{1/h}\log(x)^{1/h} \ll f(x) \ll x^{1/(h-1)} \quad \text{ and }\quad \int_{x_0}^x \frac{f(t)}{t}\mathrm{d}t = \Theta(f(x)). \]
 Since $f(x)\ll x$, let $K := \limsup_x f(x)/x$. Choose $\epsilon>0$ and let $x^{*}$ be such that $f(x) \leq (K+\epsilon)x$ for all $x\geq x^{*}$. Define then, for every $n\geq 0$,
 \[ \begin{cases} \alpha_n := 1,\negphantom{$\alpha_n = 1,$}\phantom{\alpha_n = (K+\epsilon)^{-1} f(n)/n,} \text{ if $n < x^{*}$;} \\ \alpha_n := (K+\epsilon)^{-1} f(n)/n, \text{ if $n\geq x^{*}$.}\end{cases}  \]
 By Lemma \ref{pilem}, $f$ is in particular almost increasing, thus
 \[ \int_{x_0}^{x} \frac{f(t)}{t}\mathrm{d}t \asymp \sum_{x_0 <n\leq x} \frac{f(n)}{n}, \]
 which, by the same lemma, is $\Theta(f(x))$. From strong law of large numbers it then follows that in $\langle\mathcal{S},(\alpha_n)_n \rangle$ it holds $|\omega\cap[0,x]| \uptxt{a.s.}{\asymp} f(x)$. Since this is a non-empty  probability space, there must be at least some $\C\subseteq\N$ for which $C(x)=\Theta(f(x))$.
  
 Based on this, let us construct a more convenient sequence $\A$. We know from $x^{1/h}\log(x)^{1/h}\ll C(x) \ll x^{1/(h-1)}$ that there must exist $q_1,q_2 > 0$ for which $\exists n_0\in\N : \forall n\geq n_0$ the following holds:
 \setcounter{equation}{6}
 \begin{equation}
  q_1 \cdot x^{1/h}\log(x)^{1/h}\leq C(x) \leq q_2\cdot x^{1/(h-1)} \label{q1q2}
 \end{equation}
 With this in mind, take:
 \begin{equation}
  M_1,M_2 > 1 \quad \textit{large real constants to be specified later.} \label{m1m2}
 \end{equation}
 Let $n_1\in\N$ be large enough so that, for all $n\geq n_1$,
 \begin{equation*}
  \begin{split}
   M_1 n^{1/h}\log(n)^{1/h} + 1 &> M_1(n+1)^{1/h}\log(n+1)^{1/h}, \\
   \frac{1}{M_2} n^{1/(h-1)} + 1 &> \frac{1}{M_2} (n+1)^{1/(h-1)}.
  \end{split}
 \end{equation*}  
 The following construction is then well-defined:
 \begin{itemize}
  \item Define $\A_{n_1} := \{\corr{0,}1,2\ldots,n_1\}$.
   
  \item For each integer $k > n_1$, define $\A_k$ in the following manner:
  \begin{itemize}
   \item In case $k \notin \C$, check whether
   \[ \frac{|\A_{k-1}\cap [0,k]|}{k^{1/h}\log(k)^{1/h}} \leq M_1. \]
   If so, take $\A_{k} := \A_{k-1} \cup \{k\}$; otherwise, take $\A_{k} := \A_{k-1}$.
    
   \item In case $k \in \C$, check whether
   \[ \frac{|(\A_{k-1}\cup \{k\})\cap [0,k]|}{k^{1/(h-1)}} \geq \frac{1}{M_2}. \]
   If so, take $\A_{k} := \A_{k-1}$; otherwise, take $\A_{k} := \A_{k-1}\cup\{k\}$.
  \end{itemize}
 \end{itemize}
 Finally, let $\A := \bigcup_{k\geq n_1} \A_k$. From the symmetric nature of this construction and from \eqref{q1q2}, it is possible to deduce that if $M_1,M_2$ are large enough so that $M_1 M_2 > q_1/q_2$, then $\A$ must satisfy
 \begin{equation}
  \limsup_{n\to +\infty} \frac{A(n)}{C(n)} \leq \frac{M_1}{q_1} \quad \text{and} \quad \liminf_{n\to +\infty} \frac{A(n)}{C(n)} \geq \frac{1}{q_2 M_2}, \label{relbc}
 \end{equation}
 which implies $A(x) = \Theta(C(x))$. More importantly, $\A$ satisfies:
 \begin{equation}
  M_1 \cdot n^{1/h}\log(n)^{1/h}\leq A(n) \leq \frac{1}{M_2} \cdot n^{1/(h-1)} \label{impbnd}
 \end{equation}
 from some point onward.
  
 Before proceeding, we highlight the important stuff that shall be used throughout the proof of our next results in the following remark.
  
 \begin{xrem}
  If the constants $M_1,M_2$ from \eqref{m1m2} are large enough -- i.e. satisfying $M_1 M_2 > q_1/q_2$ where $q_1,q_2$ are the constants from \eqref{q1q2}, then we can construct an OR$_{+}$ sequence $\A$ with $A(n)=\Theta(f(n))$ satisfying \eqref{impbnd} for all sufficiently large $n$.
 \end{xrem}
 
 In short, to prove the Main Theorem stated at the introduction we first show that $\rhohat_{\omega,h}(n)$ has a.s. the right rate of growth, then $\rho_{\omega,h}(n)$, and then $r_{\omega,h+\ell}(n)$ for all $\ell\geq 0$.
  
 \begin{lem}\label{pholds} In the above construction of $\A$, \corr{if $M_1$ and $M_2$ are sufficiently large,} then in $\mathcal{S}_{\A}$ the following holds:
 \begin{equation*}\rhohat_{\omega,h}(n) \uptxt{a.s.}{=} \Theta\left(\frac{A(n)^h}{n}\right). \end{equation*}
 \end{lem}
 \begin{proof}
  We divide the proof into two parts. \medskip
 
  \noindent
  $\bullet \textbf{ Part 1: } \textit{Showing $\rhohat_{\omega, h}(n) \uptxt{a.s.}{\ll} A(n)^h/n$}$.
  
  Let $T>0$ be a large integer. By the disjointness lemma,
  \[ \Pr\left(\rhohat_{\omega,h}(n) \geq T\lceil \mathbb{E}(\rho_{\omega,h}(n))\rceil \right) \leq \frac{\mathbb{E}(\rho_{\omega,h}(n))^{T\lceil \mathbb{E}(\rho_{\omega,h}(n))\rceil}}{(T\lceil \mathbb{E}(\rho_{\omega,h}(n))\rceil)!}, \]
  thus, as for all $k\geq 1$ it holds $k!>k^ke^{-k}$,
  \begin{equation*}
   \Pr\left(\rhohat_{\omega,h}(n) \geq T\lceil \mathbb{E}(\rho_{\omega,h}(n))\rceil \right) \leq \left(\frac{e}{T} \frac{\mathbb{E}(\rho_{\omega,h}(n))}{\lceil \mathbb{E}(\rho_{\omega,h}(n))\rceil}\right)^{T\lceil \mathbb{E}(\rho_{\omega,h}(n))\rceil}.
  \end{equation*}
  Since $A(n) \geq M_1 n^{1/h}\log(n)^{1/h}$, in view of Lemma \ref{exct37} and Remark \ref{imprem} we have $\mathbb{E}(\rho_{\omega,h}(n)) > \alpha M_1^h \log(n)$ for all large $n$, where $\alpha>0$ is some constant (not depending on $\A$!). Thus, taking $T > \max\{e^2, 2\alpha^{-1} M_1^{-h}\}$ yields
  \[ \left(\frac{e}{T} \frac{\mathbb{E}(\rho_{\omega,h}(n))}{\lceil \mathbb{E}(\rho_{\omega,h}(n))\rceil}\right)^{T\lceil \mathbb{E}(\rho_{\omega,h}(n))\rceil} \ll \left(\frac{1}{e}\right)^{2\log(n)+o(\log(n))} \ll n^{-2+o(1)}, \]
  hence it follows from the Borel--Cantelli lemma that $\rhohat_{\omega,h}(n)\uptxt{a.s.}{\ll} A(n)^{h}/n$. \medskip
  
  \noindent
  $\bullet \textbf{ Part 2: } \textit{Showing $\rhohat_{\omega, h}(n) \uptxt{a.s.}{\gg} A(n)^h/n$}$.
  
  This part is basically Theorem 3 from Erd\H{o}s \& Tetali \cite{erdtet90} with the added nuances from our hypotheses. Recalling our definition of lower bounded representations, it suffices to show that $\restr{\rhohat_{\omega, h}(n)}{\geq n^{1/2}} \uptxt{a.s.}{\gg} A(n)^h/n$. From Lemma \ref{rbulk},
  \begin{equation}
   \mathbb{E}(\restr{\rho_{\omega, h}(n)}{\geq n^{1/2}}) = (1-o(1))\mathbb{E}(\rho_{\omega, h}(n)). \label{lbsqrt}
  \end{equation}
  Similar to what was done in Lemma \ref{upll}, given $\mathfrak{R}\in\restr{\mathrm{ER}_h(n)}{\geq n^{1/2}}$ we abuse the notation by letting ``$\mathfrak{R}$'' denote the event ``$\mathfrak{R} \in \restr{\mathrm{ER}_h(n;\omega)}{\geq n^{1/2}}$''. In addition, let $\sum^{(k)}$ denote the sum over all $k$-sets of pairwise disjoint lower bounded (by $n^{1/2}$) exact representations of $n$. Letting $t>0$ be a small real number, we then have:
  \begin{align}
   &\, \Pr\big(\exists \mathcal{C}\subseteq\restr{\mathrm{ER}_h(n;\omega)}{\geq n^{1/2}} \text{ maxdisfam}: |\mathcal{C}| \leq t\mathbb{E}(\rho_{\omega,h}(n)) \big) \nonumber \\
   {\color{red} \leq} &\sum_{1\leq k\leq t\mathbb{E}(\rho_{\omega,h}(n))} \Pr\big(\exists \mathcal{C}\subseteq\restr{\mathrm{ER}_h(n;\omega)}{\geq n^{1/2}} \text{ maxdisfam}: |\mathcal{C}| = k \big) \nonumber \\
   {\color{red} \leq} &\sum_{1\leq k\leq t\mathbb{E}(\rho_{\omega,h}(n))} ~~\sideset{}{^{(k)}}\sum_{\substack{\{\mathfrak{R}_1,\ldots,\mathfrak{R}_k\}\\ \text{pairwise}\\ \text{disjoint}}} \bigg(\Pr\left( \mathfrak{R}_1 \wedge\cdots \wedge\mathfrak{R}_k \right)\cdot \Pr(\{\mathfrak{R}_1\ldots\mathfrak{R}_k\} \text{ is maximal})\bigg) \nonumber
  \end{align}
  
  We may estimate both factors in the above summation separately. The first term can be estimated directly from the disjointness lemma:
  \begin{equation}
   \sideset{}{^{(k)}}\sum_{\substack{\{\mathfrak{R}_1,\ldots,\mathfrak{R}_k\}\\ \text{pairwise}\\ \text{disjoint}}} \Pr\left( \mathfrak{R}_1 \wedge\cdots \wedge\mathfrak{R}_k \right) \leq \frac{\mathbb{E}(\rho_{\omega,h}(n))^k}{k!} \label{fact1}
  \end{equation}
  For the second term we use the correlation inequality! Given $\mathfrak{R}\in \restr{\mathrm{ER}_h(n)}{\geq n^{1/2}}$, let ``$\bar{\mathfrak{R}}$'' denote the event ``$\mathfrak{R}\notin \restr{\mathrm{ER}_h(n;\omega)}{\geq n^{1/2}}$''. Then:
  \[ \Pr(\{\mathfrak{R}_1,\ldots,\mathfrak{R}_k\} \text{ is maximal}) = \Pr\Bigg(\bigwedge_{\substack{\mathfrak{R}\in \restr{\mathrm{ER}_h(n)}{\geq n^{1/2}} \\  \mathfrak{R} \cap \mathfrak{R}_j = \varnothing \, (j=1,\ldots,k) }}\bar{\mathfrak{R}}\Bigg) \]
  When $h=2$ the events $\bar{\mathfrak{R}}$ are mutually independent, and when $h\geq 3$ we can choose $n$ large enough so that $\Pr(\mathfrak{R})\leq 1/2$ for every $\mathfrak{R}\in \restr{\mathrm{ER}_h(n)}{\geq n^{1/2}}$.\footnote{This passage illustrates the necessity of considering $\restr{\mathrm{ER}_h(n)}{\geq n^{1/2}}$ instead of $\mathrm{ER}_h(n)$.} Therefore we can apply the correlation inequality as follows:
  \begin{align*}
   \Pr\Bigg(\bigwedge_{\substack{\mathfrak{R}\in \restr{\mathrm{ER}_h(n)}{\geq n^{1/2}} \\  \mathfrak{R} \cap \mathfrak{R}_j = \varnothing \, (j=1,\ldots,k) }}\bar{\mathfrak{R}}\Bigg) &\leq \Bigg(\prod_{\substack{\mathfrak{R}\in \restr{\mathrm{ER}_h(n)}{\geq n^{1/2}} \\  \mathfrak{R} \cap \mathfrak{R}_j = \varnothing \, (j=1,\ldots,k) }} \Pr(\bar{\mathfrak{R}})\Bigg) e^{2\Delta(n)} \\   
   &= \Bigg(\prod_{\substack{\mathfrak{R}\in \restr{\mathrm{ER}_h(n)}{\geq n^{1/2}} \\  \mathfrak{R} \cap \mathfrak{R}_j = \varnothing \, (j=1,\ldots,k) }} (1 - \Pr(\mathfrak{R}))\Bigg) e^{2\Delta(n)} \\
   &\leq \exp\Bigg(-\sum_{\substack{\mathfrak{R}\in \restr{\mathrm{ER}_h(n)}{\geq n^{1/2}} \\  \mathfrak{R} \cap \mathfrak{R}_j = \varnothing \, (j=1,\ldots,k) }} \Pr(\mathfrak{R})\Bigg) e^{2\Delta(n)},
  \end{align*}
  where $\Delta(n)$ can be taken to be
  \[ \Delta(n) = \sum_{\substack{\mathfrak{R},\mathfrak{R}'\in \mathrm{ER}_h(n) \\ 1\leq |\mathfrak{R}\cap\mathfrak{R'}|\leq h-2}} \Pr\left( \mathfrak{R} \wedge \mathfrak{R}'\right). \]
  \corr{However, since $A(x) \leq M_2^{-1} x^{1/(h-1)}$, we know from Lemma \ref{upll} that we may write $e^{2\Delta(n)}< \exp(2\eta\, \mathbb{E}(\rho_{\omega,h}(n)))$ for some $\eta>0$.} For the term inside $\exp$, from \eqref{lbsqrt},
  \begin{align*}
   \sum_{\substack{\mathfrak{R}\in \restr{\mathrm{ER}_h(n)}{\geq n^{1/2}} \\  \mathfrak{R} \cap \mathfrak{R}_j = \varnothing \, (j=1,\ldots,k) }} \Pr(\mathfrak{R}&) \geq \mathbb{E}(\restr{\rho_{\omega,h}(n)}{\geq n^{1/2}}) - \mathbb{E}\left(\sum_{\ell\in \bigcup\{\mathfrak{R}_1,\ldots,\mathfrak{R}_k\}} \rho_{\omega,h-1}(n-\ell)\right) \\
   &= (1-o(1))\mathbb{E}(\rho_{\omega,h}(n)) - \sum_{\ell\in \bigcup\{\mathfrak{R}_1,\ldots,\mathfrak{R}_k\}} \mathbb{E}(\rho_{\omega,h-1}(n-\ell)).
  \end{align*}
  Since $A(n) \ll n^{1/(h-1)}$, we have by Lemma \ref{exct37} that $\mathbb{E}(\rho_{\omega,h-1}(n)) = O(1)$, therefore the summation in the last equality above is bounded above by $k\cdot hJ$ for some fixed $J>0$. Since $h$,$J$ are fixed and $k\leq t\mathbb{E}(\rho_{\omega,h}(n))$, we finally arrive at {\color{red}
  \begin{equation}
   \Pr(\{\mathfrak{R}_1,\ldots,\mathfrak{R}_k\} \text{ is maximal}) \leq e^{-(1- 2\eta -\eps_t -o(1))\mathbb{E}(\rho_{\omega,h}(n))} \label{fact2}
  \end{equation}}
  for all $k$-sets $\{\mathfrak{R}_1,\ldots,\mathfrak{R}_k\}$ considered, where $\eps_t := t\cdot hJ$, which goes to $0$ as $t\to 0^{+}$.
  
  Keeping in mind that, again, $k!>k^ke^{-k}$ for all $k\geq 1$, plugging \eqref{fact1} and \eqref{fact2} into what we originally wanted to estimate yields
  \begin{align} \Pr&\big(\exists \mathcal{C}\subseteq\restr{\mathrm{ER}_h(n;\omega)}{\geq n^{1/2}} \text{ maxdisfam}: |\mathcal{C}| \leq t\mathbb{E}(\rho_{\omega,h}(n)) \big) \nonumber \\
  &\leq {\color{red} e^{-(1- 2\eta - \eps_t -o(1))\mathbb{E}(\rho_{\omega,h}(n))} } \sum_{1\leq k\leq t\mathbb{E}(\rho_{\omega,h}(n))} \frac{\mathbb{E}(\rho_{\omega,h}(n))^k}{k!} \nonumber \\
  &\leq {\color{red} e^{-(1- 2\eta - \eps'_t -o(1))\mathbb{E}(\rho_{\omega,h}(n))}} \sum_{1\leq k\leq t\mathbb{E}(\rho_{\omega,h}(n))} \left(\frac{\mathbb{E}(\rho_{\omega,h}(n))}{k}\right)^k, \nonumber
  \end{align}
  where $\eps'_t := \eps_t + t$, still going to $0$ with $t$. As a function of $k$, one can show that the maximum of $(\mathbb{E}(\rho_{\omega,h}(n))/k)^k$ is attained when $k = \mathbb{E}(\rho_{\omega,h}(n))/e$. Thus, choosing $t < 1/e$, we have
  \begin{align*}
   \sum_{1\leq k\leq t\mathbb{E}(\rho_{\omega,h}(n))} \left(\frac{\mathbb{E}(\rho_{\omega,h}(n))}{k}\right)^k &\leq t\mathbb{E}(\rho_{\omega,h}(n))\cdot t^{-t\mathbb{E}(\rho_{\omega,h}(n))} \\
   &\leq t\mathbb{E}(\rho_{\omega,h}(n))\cdot e^{-t\log(t)\mathbb{E}(\rho_{\omega,h}(n))},
  \end{align*}
  leading us to, in view of $\mathbb{E}(\rho_{\omega,h}(n)) \ll n^{\frac{1}{h-1}}$ (Lemma \ref{exct37}),
  \begin{align*}
   \Pr\big(\exists \mathcal{C}\subseteq\restr{\mathrm{ER}_h(n;\omega)}{\geq n^{1/2}} \text{ m.d.f.}\footnotemark &: |\mathcal{C}| \leq t\mathbb{E}(\rho_{\omega,h}(n)) \big) \\
   &\ll n^{\frac{1}{h-1}} {\color{red} e^{-(1-2\eta - \eps''_t -o(1))\mathbb{E}(\rho_{\omega,h}(n))} },
  \end{align*}
  where\footnotetext{m.d.f. = maxdisfam.} $\eps''_t := \eps'_t - t\log(t)$, which once more goes to $0$ with $t$.
  
  Repeating the argument at the end of Part 1, we know that for some $\beta>0$ independent of $\A$ we have $\mathbb{E}(\rho_{\omega,h}(n)) > \beta M_1^{h} \log(n)$ for all large $n$. Therefore, taking $t>0$ small, \corr{\emph{choosing $M_1>1$} large and $\eta$ small (hence, $M_2$ large enough, by Lemma \ref{upll}) so that $(1-2\eta-\eps''_t)\beta M_1^{h} > 2+\frac{1}{h-1}$,} we get
  \[ \Pr\big(\exists \mathcal{C}\subseteq\restr{\mathrm{ER}_h(n;\omega)}{\geq n^{1/2}} \text{ maxdisfam}: |\mathcal{C}| \leq t\mathbb{E}(\rho_{\omega,h}(n)) \big) \ll n^{-2+o(1)}, \]
  from which we apply the Borel--Cantelli lemma to conclude that $\restr{\rhohat_{\omega, h}(n)}{\geq n^{1/2}} \uptxt{a.s.}{\gg} A(n)^h/n$, thus concluding our proof.
 \end{proof}
 
 Since $\rhohat_{\omega,h}(n) \ll \rho_{\omega,h}(n)$, to prove that $\rho_{\omega,h}(n)$ has a.s. the right rate of growth we just need to show that
 \[ \rho_{\omega,h}(n) \uptxt{a.s\corr{.}}{\ll} \rhohat_{\omega,h}(n), \]
 and this is what the next two lemmas will be aiming at. This is immediate for $h=2$, thus the assumption $h\geq 3$ will be implicit. From Lemma \ref{o1rhat} we know that if $\mathcal{C}_n \subseteq \mathrm{ER}_{\corr{h}}(n;\omega)$ is a maximal disjoint family of exact $h$-representations of $n$ then
 \begin{equation}
  \rho_{\omega,h}(n) \uptxt{a.s.}{\ll} \sum_{k\in\bigcup\mathcal{C}_n} \rhohat_{\omega,h-1}(n-k). \label{estno1}
 \end{equation}
 In this lemma, we use the fact that $\mathbb{E}(\rho_{\omega,h-1}(n)) = O(1)$ to deduce that $\rhohat_{\omega,h-1}(n) \uptxt{a.s.}{=} O(\log(n)/\log\log(n))$. This fact, however, is not sufficient to conclude from \eqref{estno1} the statement we want. We can, however, draw this conclusion for certain partitions of $\N$. Letting $0<\eps< h^{-2}$, consider
 \begin{equation*}
  \mathbf{P}_1 := \left\{ n\in\N : A(n) \leq n^{1/(h-1)-\eps} \right\} \text{ and }~ \mathbf{P}_2 := \N\setminus \mathbf{P}_1.
 \end{equation*} 
  
 Before studying these partitions, let us first split $\mathrm{ER}_h(n;\omega)$ in view of Lemma \ref{rbulk}. We know that $\A$ is OR$_{+}$, thus its lower Matuszewska index $\mlow(A)$ is positive. With that in mind, fix some $0 < \gamma < \mlow(A)$ and take $\delta >0$ with
 \begin{equation}
  \delta < \frac{\eps (h-1)}{1-\gamma (h-1)}. \label{dltchc}
 \end{equation}
 Notice that the RHS is positive, for since $A(x)\ll x^{1/(h-1)}$ one must have $\mlow(A) \leq 1/(h-1)$. Now, split $\rho_{\omega,h}$ into $\rho_{\omega,h}(n) = \restr{\rho_{\omega,h}(n)}{\geq n^{1-\delta}} + \rho_{\omega,h}(n)\flat_{n^{1-\delta}}$, with $\restr{\mathcal{C}_{n}}{\geq n^{1-\delta}}$ and $\mathcal{C}_{n}\flat_{n^{1-\delta}}$ being maxdisfams in each corresponding set. The next lemma shows that $\rhohat_{\omega,h}$ will almost surely dominate $\rho_{\omega,h}$ in $\mathbf{P}_1$.
 
 \begin{lem}\label{p1ph}
  Let ``$\ll_1$'' denote ``$\ll$ as $n\to +\infty$ through $\mathbf{P}_1$''. In the above notation, following Lemma \ref{pholds} \corr{(i.e. having $M_1, M_2 >1$ sufficiently large),} if $\mathbf{P}_1$ is infinite then in $\mathcal{S}_{\A}$ the following holds:
  \[ \rho_{\omega,h}(n) \uptxt{a.s.}{\ll}_1 \rhohat_{\omega,h}(n). \]
 \end{lem}
 \begin{proof}
  We divide the proof into two parts. First, we show that $\rho_{\omega,h}(n)\flat_{n^{1-\delta}}$ does not contribute to the growth rate of $\rho_{\omega,h}(n)$. \medskip
  
  \noindent
  $\bullet \textbf{ Part 1:}$ $\textit{Showing}$ $\rho_{\omega,h}(n)\flat_{n^{1-\delta}} \uptxt{a.s.}{=} O\left(\rhohat_{\omega,h}(n)/\log\log(n) \right)$.
  
  Recall that this quantity is counting those representations in which at least one member is less than $n^{1-\delta}$. From \eqref{clncut} in Lemma \ref{rbulk},
  \begin{equation*}
   \mathbb{E}(\rho_{\omega,h}(n)\flat_{n^{1-\delta}} ) \ll n^{-\delta \gamma/2^{h-2}} \mathbb{E}(\rho_{\omega,h}(n) ),
  \end{equation*}
  thus, proceeding as in Part 1 from Lemma \ref{pholds}, we apply the disjointness lemma to estimate $|\bigcup \mathcal{C}_n\flat_{n^{1-\delta}}|$. Keeping in mind that for $k\geq 1$ it holds $k!>k^ke^{-k}$, for any integer $T>0$ it follows that
  \begin{align*}
   \Pr\left( \rhohat_{\omega,h}(n)\flat_{n^{1-\delta}} \geq T\left\lceil\frac{ {\color{red} \mathbb{E}(\rho_{\omega,h}(n))} }{\log(n)}\right\rceil \right) &\leq \frac{\left( n^{-\delta \gamma/2^{h-2}}\mathbb{E}(\rho_{\omega,h}(n)) \right)^{T\left\lceil \frac{{\color{red} \mathbb{E}(\rho_{\omega,h}(n))}}{\log(n)} \right\rceil}}{(T\lceil {\color{red} \mathbb{E}(\rho_{\omega,h}(n))}/\log(n) \rceil )!} \\
   &\leq \left(\frac{e}{T} \frac{\mathbb{E}(\rho_{\omega,h}(n))}{{\color{red} \mathbb{E}(\rho_{\omega,h}(n))}} \frac{n^{-\delta \gamma/2^{h-2}}}{1/\log(n)} \right)^{T\left\lceil \frac{{\color{red} \mathbb{E}(\rho_{\omega,h}(n))}}{\log(n)} \right\rceil}.
  \end{align*}
  
  Since $\rhohat_{\omega,h}(n) \uptxt{a.s.}{=} \Theta(\mathbb{E}(\rho_{\omega,h}(n)))$ by Lemma \ref{pholds} and $\mathbb{E}(\rho_{\omega,h}(n)) \gg \log(n)$ by Lemma \ref{exct37}, let $\vartheta>0$ such that $\log(n)n^{-\delta\gamma/2^{h-2}} \ll n^{-\vartheta}$ and consider $T$ large enough so as to cancel out all constant terms inside the parentheses. This will yield
  \begin{align*}
   \left(\frac{e}{T} \frac{\mathbb{E}(\rho_{\omega,h}(n))}{{\color{red} \mathbb{E}(\rho_{\omega,h}(n))}} \frac{n^{-\delta \gamma/2^{h-2}}}{1/\log(n)} \right)^{T\left\lceil \frac{{\color{red} \mathbb{E}(\rho_{\omega,h}(n))}}{\log(n)} \right\rceil} &\ll \left(\log(n) n^{-\delta \gamma/2^{h-2}}  \right)^{T\left\lceil \frac{{\color{red} \mathbb{E}(\rho_{\omega,h}(n))}}{\log(n)} \right\rceil} \\
   &\ll n^{-\vartheta T\left\lceil \frac{{\color{red} \mathbb{E}(\rho_{\omega,h}(n))}}{\log(n)} \right\rceil}.
  \end{align*}
  Therefore, picking $T$ large enough to also have $\vartheta T\lceil {\color{red} \mathbb{E}(\rho_{\omega,h}(n))}/\log(n) \rceil > 2$ from some point onward, we conclude that $\rhohat_{\omega,h}(n)\flat_{n^{1-\delta}} \uptxt{a.s.}{\ll} \rhohat_{\omega,h}(n)/\log(n)$ from the Borel--Cantelli lemma. Finally, the same reasoning behind \eqref{estno1} may be used to derive
  \begin{equation*}
   \rho_{\omega,h}(n)\flat_{n^{1-\delta}} \uptxt{a.s.}{\ll} \sum_{k\in\bigcup \mathcal{C}_n\flat_{n^{1-\delta}}} \rhohat_{\omega,h-1}(n-k)\flat_{n^{1-\delta}},
  \end{equation*}
  thus, since $\rhohat_{\omega,h-1}(n) \uptxt{a.s.}{=} O(\log(n)/\log\log(n))$ from Lemma \ref{o1rhat}, we get:
  \begin{align*}
   \rho_{\omega,h}(n)\flat_{n^{1-\delta}} &\uptxt{a.s.}{\ll} \left|\bigcup\mathcal{C}_n\flat_{n^{1-\delta}}\right| \cdot \max_{k\in \bigcup\mathcal{C}_n\flat_{n^{1-\delta}}} \bigg(\rhohat_{\omega,h-1}(n-k)\flat_{n^{1-\delta}}\bigg) \\
   &\uptxt{a.s.}{\ll} \frac{\rhohat_{\omega,h}(n)}{\log(n)} \cdot \frac{\log(n)}{\log\log(n)} = \frac{\rhohat_{\omega,h}(n)}{\log\log(n)},
  \end{align*}
  as required. \medskip
  
  \noindent
  $\bullet \textbf{ Part 2:}$ $\textit{Showing}$ $\restr{\rho_{\omega,h}(n)}{\geq n^{1-\delta}} \uptxt{a.s.}{\ll}_1 \rhohat_{\omega,h}(n)$.
  
  Recall that $\delta = \delta(\eps, \gamma)$, i.e. it depends on both $\eps$ and $\gamma$, where $0< \gamma< \mlow(A)$. By the definition of $\mlow(A)$, we know there is $m>0$ such that for large $x$, for every $\lambda \geq 1$ the following holds:
  \[ A(\lambda x) > m\lambda^{\gamma} A(x). \]
  Rearranging this, for large $x$ we have
  \[ \frac{A(x^{1-\delta})^{h-1}}{x^{1-\delta}} < \frac{1}{m^{h-1}}x^{\delta(1-\gamma (h-1))} \frac{A(x)^{h-1}}{x}, \]
  thus, in view of Lemma \ref{exct37},
  \begin{equation*}
   \mathbb{E}(\rho_{\omega,h-1}(n^{1-\delta})) \ll n^{\delta(1-\gamma(h-1))} \mathbb{E}(\rho_{\omega,h-1}(n)).
  \end{equation*}
  The partition $\mathbf{P}_1$ was chosen because $\mathbb{E}(\rho_{\omega,h-1}(n)) \ll_1 n^{-\eps(h-1)}$, and thus
  \[ \mathbb{E}(\rho_{\omega,h-1}(n^{1-\delta})) \ll_1 n^{\delta(1-\gamma(h-1)) -\eps(h-1)}. \]
  Therefore, by our choice of $\delta$ in \eqref{dltchc}, we have that for some $\vartheta >0$ it holds:
  \[ \mathbb{E}(\rho_{\omega,h-1}(k)) \ll k^{-\vartheta} \text{ as } k\to +\infty \text{ through } \mathbf{P}_1 \cup \{[n^{1-\delta}, n] : n \in \mathbf{P}_1 \}, \]
  
  Using the same strategy used in items \eqref{l51is}* and \eqref{l51iis}* at the proof of Lemma \ref{o1rhat}, our argument to bound $\restr{\rho_{\omega,h}(n)}{\geq n^{1-\delta}}$ will be an application of the disjointness lemma followed by Borel--Cantelli. From disjointness, for any integer $T>2/\vartheta$ we will have
  \begin{equation*}
   \Pr\left(\rhohat_{\omega,h-1}(k) \geq T \right) \leq \frac{\mathbb{E}(\rho_{\omega,h-1}(k))^{T}}{T!} \ll k^{-2}
  \end{equation*}
  as $k\to +\infty$ through $\mathbf{P}_1 \cup \{[n^{1-\delta}, n] : n \in \mathbf{P}_1\}$; by the Borel--Cantelli lemma, this implies $\rhohat_{\omega,h-1}(k) \uptxt{a.s.}{=} O(1)$ as $k$ varies throughout this domain.
  
  Now, just as in \eqref{estno1}, we have  
  \begin{equation}
   \restr{\rho_{\omega,h}(n)}{\geq n^{1-\delta}} \uptxt{a.s.}{\ll} \sum_{k\in\bigcup \restr{\mathcal{C}_n}{\geq n^{1-\delta}}} \restr{\rhohat_{\omega,h-1}(n-k)}{\geq n^{1-\delta}}. \label{c1e1p1}
  \end{equation}  
  Since the restrictions on the RHS imply that the corresponding sum is occurring only over $n-k \geq n^{1-\delta}$, and also since $\rhohat_{\omega,h-1}(k)$ is conveniently limited on the domain needed, it follows that
  \[ \restr{\rho_{\omega,h}(n)}{\geq n^{1-\delta}} \uptxt{a.s.}{\ll}_1 \left| \bigcup\restr{\mathcal{C}_{n}}{\geq n^{1-\delta}} \right|, \]
  which is obviously $O(\rhohat_{\omega,h}(n))$.
 \end{proof}

 {\color{red} \begin{rem}[Correction to the published version]\label{rem-corr}
  The argument occupying this place in the published version, including Lemma 5.5, contained a gap.  The set $\widetilde{\mathcal B}^{(T)}_n$ used there was defined conditionally, so the disjointness lemma could not be applied to $\widehat{\beta}^{(T)}(n)$.  Under the corrected hypothesis $A(n)\ll n^{1/(h-1)-\eps}$, the set $\mathbf{P}_1$ is cofinite, and Lemma \ref{p1ph} gives
  \[ \rho_{\omega,h}(n) \uptxt{a.s.}{\asymp} \rhohat_{\omega,h}(n) \uptxt{a.s.}{\asymp} \frac{A(n)^h}{n}. \]
  Together with Lemma \ref{o1rhat}, this implies
  \[ r_{\omega,h}(n) \uptxt{a.s.}{\asymp} \frac{A(n)^h}{n}. \]
  The full range claimed in the published version, and in fact a slightly larger range, can be recovered by combining \cite[Theorem 1.4]{taf25} with Theorem \ref{p314} below.
 \end{rem} }

 It then follows immediately from \eqref{loll2} in Lemma \ref{o1rhat} that $r_{\omega,h}(n)$ has a.s. the right rate of growth, and the case $\ell=0$ of the Main Theorem follows immediately.\footnote{More precisely, it follows from the fact that $\mathcal{S}_{\A}$ is a non-empty probability space.} The next result shall deal with the further cases. We actually show something a little more general: if one could prove that $r_{\omega,h}(n)$ has a.s. the right growth order, then, provided $M_1$,$M_2$ are large enough, $r_{\omega,h+\ell}(n)$ will also concentrate around its mean for all $\ell \geq 0$.
  
 \begin{thm}\label{p314}
  Suppose there is $M^{\dagger} \in\R_{+}$ such that, in the construction of $\A$, if $M_1,M_2>M^{\dagger}$ then in $\mathcal{S}_{\A}$ the following holds:
  \[ \rho_{\omega,h}(n) \uptxt{a.s.}{\ll} \rhohat_{\omega,h}(n). \]
  Then there must be $M^{\ddagger}\in\R_{+}$ such that $M_1,M_2>M^{\ddagger}$ implies
  \[ r_{\omega,h+\ell}(n) \uptxt{a.s.}{=} \Theta\left(\frac{A(n)^{h+\ell}}{n}\right),\quad \forall \ell \geq 0 \]
 \end{thm}
 \begin{proof}
  Let us start with the observation that if $t_1,t_2\geq h$ are such that
  \[ r_{\omega,t_1}(n) \uptxt{a.s.}{=} \Theta\left(\frac{A(n)^{t_1}}{n}\right) \text{ and } r_{\omega,t_2}(n) \uptxt{a.s.}{=} \Theta\left(\frac{A(n)^{t_2}}{n}\right), \]
  then $r_{\omega,t_1 +t_2}(n) \uptxt{a.s.}{=} \Theta(A(n)^{t_1+t_2}/n)$. In fact, from the recursive formulas that can be derived from \eqref{forms}, it follows that
  \begin{align*} r_{\omega,t_1+t_2}(n) &= \sum_{k\leq n} r_{\omega,t_1}(k)r_{\omega,t_2}(n-k) \\
  &\uptxt{a.s.}{\asymp} \sum_{1\leq k\leq n-1} \frac{A(k)^{t_1}}{k}\frac{A(n-k)^{t_2}}{n-k} \\
  &= \sum_{1\leq k\leq n/2} \frac{A(k)^{t_1}}{k}\frac{A(n-k)^{t_2}}{n-k} + \sum_{n/2 < k\leq n-1} \frac{A(k)^{t_1}}{k}\frac{A(n-k)^{t_2}}{n-k},\end{align*}
  which may be dealt with in the same way we treated \eqref{keyobs}. Hence, for instance, if $M_1,M_2>M^{\dagger}$ then $r_{\omega,th}(n) \uptxt{a.s.}{=} \Theta(A(n)^{th}/n)$ for every $t\geq 1$, for the case $t=1$ follows from our hypothesis. With this in mind, to prove our theorem it is sufficient to show that $r_{\omega,h+\ell}(n) \uptxt{a.s.}{=} \Theta(A(n)^{h+\ell}/n)$ for $0\leq \ell < h$, thus the remaining cases will follow from this observation.
  
  We know this is true for $\ell = 0$, so we shall proceed by induction on $\ell$. Assuming it is valid for some $\ell \in \{0,1,\ldots,h-2\}$, we will deduce the case $\ell+1$. Start by noticing that
  \setcounter{equation}{20} % Preserve the equation numbering of the published version.
  \begin{align}
   r_{\omega,h+\ell+1}(n) &= \sum_{k\leq n} r_{\omega,h+\ell}(k)\mathbbm{1}_{\omega}(n-k) \nonumber \\
   &\uptxt{a.s.}{\asymp} \sum_{1\leq k\leq n} \frac{A(k)^{h+\ell}}{k}\mathbbm{1}_{\omega}(n-k) \nonumber \\
   &= \sum_{T<k\leq n-T} \frac{A(k)^{h+\ell}}{k}\mathbbm{1}_{\omega}(n-k) + O\left(\frac{A(n)^{h+\ell}}{n}\right), \label{finalboss}
  \end{align}
  where $T >0$ is some arbitrarily large constant to be determined. The main term from \eqref{finalboss} is a sum of independent random variables bounded by $\max_{T< k \leq n-T} \{A(k)^{h+\ell}/k\}$, therefore, letting
  \[ X_n^{(\ell)} := \sum_{T<k\leq n-T} \frac{A(k)^{h+\ell}}{k}\mathbbm{1}_{\omega}(n-k), \]
  we can apply the Chernoff bounds stated in Section \ref{sec1} to obtain:
  \begin{equation}
   \Pr\left(\left|\frac{X_n^{(\ell)}}{\mathbb{E}(X_n^{(\ell)})} -1\right| \geq \frac{1}{2}\right) \leq 2e^{-\frac{1}{16}\mathbb{E}(X_n^{(\ell)}) \left(\max\limits_{T< k \leq n-T} \frac{A(k)^{h+\ell}}{k}\right)^{-1}} \label{crazychern}
  \end{equation}
  Here enters our final step: to bound from below the absolute value of the exponent in \eqref{crazychern}. Since we know from the Lemma \ref{t37} that $\mathbb{E}(X_n^{(\ell)}) = \Theta(A(n)^{h+\ell+1}/n)$, it suffices to show this exponent is at least $2\log(n) + o(\log(n))$, then we may apply the Borel--Cantelli lemma and our theorem will follow.
  
  We start by bounding $\mathbb{E}(X_n^{(\ell)})$ from below. We have
  \begin{align*}
   \mathbb{E}(X_n^{(\ell)}) &= \sum_{T< k\leq n-T} \frac{A(k)^{h+\ell}}{k}\mathbb{E}(\mathbbm{1}_{\omega}(n-k)) \\
   &\geq \sum_{n/2< k\leq n-T} \frac{A(k)^{h+\ell}}{k}\frac{A(n-k)}{n-k+1} \\
   &\geq \frac{A(n/2)^{h+\ell}}{n} \sum_{n/2< k\leq n-T} \frac{A(n-k)}{n-k+1}
  \end{align*}
  Since $\A$ is OR$_{+}$, from Lemma \ref{pilem} we know there is some constant $d>0$ for which
  \[ \mathbb{E}(X_n^{(\ell)}) \geq d \frac{A(n)A(n/2)^{h+\ell}}{n} \quad (\text{$\forall$large $n$}). \]
  Note that $d$ does not depend on $\ell$. More importantly, it also does not depend on $T$, as can be seen from \eqref{finalboss}. Now it remains us to bound the ``$\max$'' term from above. Firstly, since $A(n)$ is non-decreasing, we have
  \[ \left(\max\limits_{T< k \leq n-T} \frac{A(k)^{h+\ell}}{k}\right) \leq A(n)^{\ell + 1}\cdot \left(\max\limits_{T< k \leq n-T} \frac{A(k)^{h-1}}{k}\right). \]
  Since $A(n) \leq M_2^{-1} n^{1/(h-1)}$ for large $n$, we may take $T$ large enough so that
  \[ \left(\max_{T< k \leq n-T} \frac{A(k)^{h-1}}{k}\right) < \frac{2}{M_2^{h-1}}. \]
  
  In view of these estimates, we then have, for all large $n$,
  \begin{equation}
   \frac{1}{16} \mathbb{E}(X_n^{(\ell)})\left(\max\limits_{T< k \leq n-T} \frac{A(k)^{h+\ell}}{k}\right)^{-1} > d\frac{M_2^{h-1}}{32} \frac{A(n/2)^\ell}{A(n)^{\ell}}\frac{A(n/2)^{h}}{n}. \label{crazychern2}
  \end{equation}
  There are two problematic terms in this lower bound: $A(n/2)^{\ell}/A(n)^{\ell}$ \& $A(n/2)^{h}/n$. For the first one we have to go back to our original OR$_{+}$ sequence $\C$. Since $\C$ is in particular OR, we know there exists some small $\vartheta > 0$ for which $C(n/2)/C(n) > \vartheta$ for all large $n$. We also know from \eqref{relbc} that for every $\eps > 0$ there is some $n_\eps \in \N$ such that, for all $n > n_{\eps} $,
  \begin{align*}
   C(n/2) &< (q_2 M_2 + \eps) A(n/2) \\
   C(n) &> \left(\frac{q_1}{M_1}-\eps\right) A(n),
  \end{align*}
  being $q_1,q_2$ the constants from the construction of $\C$. Putting these two equations together we get that
  \[ \frac{A(n/2)}{A(n)} > \frac{q_1}{q_2}\frac{\vartheta }{M_1 M_2} \]
  for all large $n$, thus, since $\ell$ ranges from $0$ to $h-2$, we have the first problematic term bounded from below by $(M_1 M_2)^{-(h-2)} (\frac{q_1}{q_2} \vartheta)^{h-2}$. The lower bound to the second term being considered, $A(n/2)^{h}/n$, comes from the fact that $A(n) \geq M_1 n^{1/h}\log(n)^{1/h}$ for large $n$. This implies
  \[ \frac{A(n/2)^{h}}{n} \geq \frac{M_1^h}{2}\log(n) + O(1). \]
  Finally, \eqref{crazychern2} then turns into
  \begin{equation}
   d\frac{M_2^{h-1}}{32} \frac{A(n/2)^\ell}{A(n)^{\ell}}\frac{A(n/2)^{h}}{n} > M_1^{2} M_2 \frac{d}{64} \left(\frac{q_1}{q_2}\vartheta \right)^{h-2}\log(n) + O(1). \label{crazychern3}
  \end{equation}
  
  In conclusion, having $M_1$,$M_2$ large enough so that
  \[ M_1^2 M_2 > \frac{2}{\frac{d}{64} \left(\frac{q_1}{q_2}\vartheta \right)^{h-2}} \]
  will imply that \eqref{crazychern2} is greater than $2\log(n) + o(\log(n))$, so we may finally apply the Borel--Cantelli lemma in \eqref{crazychern} to obtain, in view of \eqref{finalboss},
  \[ r_{\omega,h+\ell+1}(n) \uptxt{a.s.}{\asymp} \frac{A(n)^{h+\ell+1}}{n}, \]
  and this concludes our proof.
 \end{proof}
 
 Since the countable intersection of events with probability $1$ also has probability $1$, our Main Theorem follows from Lemmas \ref{pholds}, \ref{p1ph}, \corr{Remark \ref{rem-corr},} Theorem \ref{p314} and the fact that $\mathcal{S}_{\A}$ is non-empty.
  
 \section{Further remarks}\label{sec5}  
 \begin{rem}
  On the conditions of Theorem \ref{p314}, one can extend the result from Lemma \ref{o1rhat} by showing that
  \[ \binom{h+\ell}{2}\sum_{k\leq n/2} r_{\omega,h+\ell-2}(n-2k)\mathbbm{1}_{\omega}(k) \uptxt{a.s.}{\ll} A(n)^{\ell}\frac{\log(n)}{\log\log(n)}, \]
  where the LHS bounds the number of non-exact $(h+\ell)$-representations of $n$ from above. Since we have $r_{\omega,h+\ell}(n) \uptxt{a.s.}{\gg} A(n)^{\ell}\,\corr{\log(n)}$, it then follows that $\rho_{\omega,h+\ell}(n) \uptxt{a.s.}{\asymp} r_{\omega,h+\ell}(n)$. This fact implies that the Main Theorem must also apply to exact representation functions. Originally, Erd\H{o}s and Tetali \cite{erdtet90} worked exclusively with $\rho_{\omega,h}$.
 \end{rem}
 
 \begin{rem}
  Given two sequences $\A,\mathscr{L}\subseteq \N$, one may define their \emph{composition} as $\mathscr{L}[\A] := \{\ell_{a_0},\ell_{a_1},\ell_{a_2},\ldots\}$. Since $|\mathscr{L}[\A]\cap [0,x]| = A(L(x)-1)$, we can define, in view of Step \ref{stepII} (Section \ref{sec3}), the \emph{space of $\mathscr{L}$-subsequences} $\mathscr{L}[\mathcal{S}_{\A}]$ as being $\langle \mathcal{S},(\alpha_n)_n \rangle$ with 
  \[ \alpha_n := \frac{A(L(n)-1)}{L(n)}\mathbbm{1}_{\mathscr{L}}(n). \]
  This is just a more direct way of considering $\mathscr{L}[\omega]$ with $\omega \in \mathcal{S}_\A$. Since $\A$ is OR$_{+}$, by the strong law of large numbers (Theorem \ref{slln}) we have $|\omega \cap [0,x]| \uptxt{a.s.}{\asymp} A(L(x))$; in addition, $A(L(n))^h \frac{r_{\mathscr{L},h}(n)}{s_{\mathscr{L},h}(n)}$ is then the natural candidate for $\mathbb{E}(r_{\omega,h}(n))$.
 
  It is a natural question to ask which kinds of regularity assumptions do we need to impose on $\mathscr{L}$ to be able to draw nice conclusions from $\mathscr{L}[\mathcal{S}_{\A}]$. In the case of Waring bases, for example, one can find subbases with similar properties to those described by our Main Theorem (see Vu \cite{vvu00wp}).
 \end{rem}

%%%%%%%%%%%%%%%%%%%%%%%%%%%%%%%%%
%%%%%%%%%%%%%%%%%%%%%%%% APPENDIX
%%%%%%%%%%%%%%%%%%%%%%%%%%%%%%%%%%%%%%%%%%%%%%%%%%%%%%%%%%%%%%%%%%

\appendix
\section{\texorpdfstring{$O$}{O}-regularity in sequences}\label{appA}
 In this appendix we motivate and characterize the concepts of OR, PI and OR$_+$ sequences. In spite of OR$_+$ sequences being an essential part of our proof, this section is only tangentially related to the main subject of our paper. In Propositions \ref{chrreg}, \ref{chrpi} and \ref{chrorpi} one finds equivalent definitions of OR, PI and OR$_{+}$ sequences, resp.. In Subsection \ref{subsecA5} we show an application of these concepts in the non-probabilistic context of additive bases, culminating on the sketch of an essentially elementary proof that primes that split completely in a given number field\corr{, together with $\{0,1\}$,} constitute a basis.
 
 \begin{ntt}
  Script letters $\A$, $\B$, $\mathscr{C}$... denote \emph{sequences}, which are infinite subset of $\N$. Roman capital letters $A$, $B$, $C$... denote the \emph{counting function} of the sequence represented by the corresponding script letter, as in $A(x) = |\A\cap [0,x]|$. Lowercase roman letters $a_0$, $a_1$, $a_2$... (resp. $b$, $c$...) denote the elements of $\A$ (resp. $\B$, $\mathscr{C}$...) in order, with $a_0$ (resp. $b_0$, $c_0$...) being its smallest element.
 \end{ntt}

\subsection{Generalities: from \texorpdfstring{$A(x)$}{A(x)} to \texorpdfstring{$s_{\A,h}(x)$}{s\_A,h(x)}}
 One of the overarching goals in the study of additive bases is to \emph{characterize all bases}. This is a problem with clear number-theoretical grounding, as in this direction we find the study of the \emph{classical bases}: primes (Goldbach's conjecture), $k$th-powers (Waring's problem) and polygonal numbers (Cauchy--Fermat polygonal number theorem). A thorough overview of these topics is covered by Nathanson \cite{nathanson96}.
 
 One, however, might argue that this is more naturally understood as a \emph{combinatorial} problem. In this direction we find all the research springing from \emph{Sidon sequences}, such as the Erd\H{o}s--Tur\'{a}n conjecture for additive bases and the very Erd\H{o}s--Tetali theorem, as well as Erd\H{o}s--Fuchs-type estimates for representation functions. The classical reference for this more abstract treatment of sequences is Halberstam \& Roth \cite{halberstam83}, and the types of questions concerning us here are the ones outlined in the introduction of Chapter II, which we paraphrase:
 
 \begin{enumerate}[I.]
  \item What is the relationship between the asymptotic behaviour of the counting function of a sequence and its representation functions? \label{qI}
  
  \item Which properties of classical bases used to derive that these are in fact bases are actually applicable to wider classes of sequences? \label{qII}
  
  \item Which techniques employed in the study of classical bases are generalizable to a broader context? \label{qIII}
 \end{enumerate}

 In this appendix we only deal with the representation functions coming from the formal series (\ref{forms}). Given a sequence $\A\subseteq \N$, let $n\in\N$ and $h\geq 2$ an integer. The following recursive formulas can be immediately deduced:
 \begin{equation}
  \begin{split}
   r_{\A,h}(n) &= \sum_{k\leq n} r_{\A,h-\ell}(k)r_{\A,\ell}(n-k) \\
   s_{\A,h}(n) &= \sum_{k\leq n} r_{\A,h-\ell}(k)s_{\A,\ell}(n-k) = \sum_{k\leq n} s_{\A,h-\ell}(k)r_{\A,\ell}(n-k)
  \end{split} \label{p1}
 \end{equation}
 For purposes of induction the case $\ell=1$ is usually enough, and we denote $r_{\A,1}(n)$ simply by $\mathbbm{1}_{\A}(n)$. Subsection \ref{subsecA5} tries to provide an answer for questions \ref{qII} and \ref{qIII} via a framework which mixes both Schnirelmann's classical theory of sequences and Hardy--Littlewood's circle method, while the remainder is dedicated to exploring question \ref{qI} and, from it, possible ``classes of sequences'' as mentioned in question \ref{qII}.
 
 We start with the following lemma.
 
 \begin{lem}\label{ax2sax}
  For every sequence $\A\subseteq\N$,
  \begin{equation*}
   A(x/2)^h \leq s_{\A,h}(x) \leq A(x)^h, \quad \forall h\geq 1.
  \end{equation*}
 \end{lem}
 \begin{proof}
  The case $h=1$ is clear, hence we proceed by induction. By the recursive formulas in \eqref{p1},
  \begin{align*}
   s_{\A,h}(n) &= \sum_{k\leq n} s_{\A,h-1}(k)\mathbbm{1}_{\A}(n-k) \\
   &\leq s_{\A,h-1}(n)\sum_{k\leq n} \mathbbm{1}_{\A}(n-k) \\
   &= s_{\A,h-1}(n)\cdot A(n) \leq A(n)^h,
  \end{align*}
  and
  \begin{align*}
   s_{\A,h}(n) &\geq \sum_{n/2\leq k\leq n} s_{\A,h-1}(k)\mathbbm{1}_{\A}(n-k) \\
   &\geq s_{\A,h-1}(n/2)\sum_{n/2\leq k\leq n} \mathbbm{1}_{\A}(n-k)\\
   &= s_{\A,h-1}(n/2)\cdot A(n/2) \geq A(n/2)^h,
  \end{align*}
  as required.
 \end{proof}
 
 The reason why we initially focus on $s_{\A,h}(n)$ instead of $r_{\A,h}(n)$ is that, apart from being easier to deal with, it relates to $r_{\A,h}(n)$ as being a sort of an ``average'' when one tries to employ the rough approximation 
 \[ r_{\A,h}(n) \approx \frac{s_{\A,h}(n)}{n+1}. \]
 Much of what we propose comes from this intuition. To put more precisely, we seek for the conditions necessary for this intuition to make sense.

\subsection{OR sequences}\label{appA2}
 Let us introduce some bits of regular variation theory. An extensive treatment on this topic can be found in Bingham, Goldie \& Teugels \cite{bingham89}. We will only use the theory from Chapters 1 and 2.

 Take $f:[x_0,+\infty)\to \R_{+}$ a positive real function. We say that $f$ is
 \begin{itemize}
  \item \emph{Slowly varying} if $f(\lambda x) \sim f(x)$ for all $\lambda >0$;
  \item \emph{Regularly varying} if $f(\lambda x) \sim \lambda^\rho f(x)$, for all $\lambda>0$ and some $\rho\in\R$;
  \item \emph{$O$-regularly varying} if $f(\lambda x)\asymp f(x)$ for all $\lambda >0$.
 \end{itemize}

 The generality of these definitions lies on Karamata's characterization theorem,\footnote{Theorem 1.4.1, p. 17 of Bingham, Goldie \& Teugels \cite{bingham89}.} which states that if $f(\lambda x) \sim g(\lambda)f(x)$ with $g(\lambda)\in(0,+\infty)$ for all $\lambda >0$ and $f$ is measurable, then $f$ is regularly varying, i.e. there is $\rho\in\R$ such that $g(\lambda)=\lambda^\rho$. One can then promptly see how $O$-regular variation extends regular variation, and that is why we shall focus only on the former. We have mentioned in the introduction that a sequence $\A\subseteq \N$ is OR whenever $A(2x) = O(A(x))$. This is equivalent to $A$ being an $O$-regularly varying function. To be more precise, we define an \emph{OR sequence} to be a sequence satisfying any of the equivalent conditions of the following proposition.
   
 \begin{prop}[OR sequences] \label{chrreg} Let $\A\subseteq\N$ be a sequence. The following are equivalent:
 \begin{enumerate}[(i)]
  \item $A$ is $O$-regularly varying; \label{a2i}
  \item $A(2x) = O\left(A(x)\right)$; \label{a2ii}
  \item $s_{\A,h}(x) = \Theta(A(x)^h)$ for all $h\geq 1$; \label{a2iii}
  \item $s_{\A,h}(2x) = O(s_{\A,h}(x))$ for all or at least some $h\geq 1$.\footnote{By this we mean that if this holds for at least some $h\geq 1$, then it also holds for all $h\geq 1$. The same appears in items $\text{(iv)}$ of Propositions \ref{chrpi} and \ref{chrorpi}.} \label{a2iv}
 \end{enumerate}
 \end{prop}
 \begin{proof}
  \eqref{a2iii}$\implies$\eqref{a2iv} is clear, and \eqref{a2i}$\implies$\eqref{a2iii} follows from Lemma \ref{ax2sax}.

  We show now that \eqref{a2iv}$\implies$\eqref{a2ii}. The case $h=1$ is trivial, so suppose $h\geq 2$. From Lemma \ref{ax2sax} it follows $s_{\A,h}(x) \leq A(x)^h \leq s_{\A,h}(2x)$, therefore \eqref{a2iv} implies $s_{\A,h}(x)\asymp A(x)^h$. \corr{Hence}
  \[ {\color{red} A(2x)^h \asymp s_{\A,h}(2x) \ll s_{\A,h}(x) \asymp A(x)^h,} \]
  \corr{and} thus $A(2x)\ll A(x)$.
  
  Finally, to show that \eqref{a2ii}$\implies$\eqref{a2i}, note that $A(x)$ is non-decreasing, thus \eqref{a2ii} is equivalent to having $A(x)\asymp A(2x)$. Hence, for all $k\in\Z$,
  \[ A(x) \asymp A(2^k x), \]
  which implies $A(x) \asymp A(\lambda x)$ for all $\lambda>0$.
 \end{proof}

 Keeping Lemma \ref{ax2sax} in mind, one can then say that OR sequences are the most well-behaved sequences in terms of growth order, for it is sufficient to have $A(x)$ in order to deduce the growth of $s_{\A,h}$.
  
 \subsection{PI sequences}\label{appA3} 
 In spite of that, just assuming $\A$ to be OR does not guarantee that $a_{n}$, when viewed as a function of $n$, is $O$-regularly varying. For this part we need another bit of regular variation theory. Still following Bingham, Goldie \& Teugels \cite{bingham89}, a positive real function $f:[x_0,+\infty)\to\R_+$ is said to be \emph{almost increasing} when there is $m>0$ such that
 \[ f(y) \geq mf(x), \quad \forall y\geq x\geq x_0. \]
 This is equivalent to saying that $f(x) \ll \inf_{y\geq x} f(y)$; \emph{almost decreasing} functions can be defined in an analogous way, namely $f(x) \gg \sup_{y\geq x} f(y)$. Consider then the \emph{upper} ($\mupp$) and \emph{lower Matuszewska index} ($\mlow$) of $f$:
 \begin{equation}
  \begin{split}
   \mupp(f) &:= \inf\,\{\gamma \in \R: x^{-\gamma}f(x) \text{ is almost decreasing}\}, \\
   \mlow(f) &:= \sup\{\gamma \in \R: x^{-\gamma}f(x) \text{ is almost increasing}\}.
  \end{split} \label{dmatusz}
 \end{equation}
 This definition is given in view of the almost-monotonicity theorem.\footnote{Theorem 2.2.2, p. 72 of Bingham, Goldie \& Teugels \cite{bingham89}.} These indices are preserved under the asymptotic sign ``$\asymp$'', and both are finite if and only if $f$ is $O$-regularly varying.\footnote{Theorem 2.1.7, p. 71 of Bingham, Goldie \& Teugels \cite{bingham89}.} Functions with $\mlow(f)>0$ are said to have \emph{positive increase}. Our motivation to study these indices comes from the next two lemmas. By ``$a_{\lfloor \cdot\rfloor}$'' we shall be denoting the positive real function that takes $x \mapsto a_{\lfloor x\rfloor}$.
  
 \begin{lem}\label{abab}
  For every sequence $\A\subseteq\N$,
  \[ \mlow(A) = \frac{1}{\mupp(a_{\lfloor \cdot\rfloor})}, \qquad \mupp(A) = \frac{1}{\mlow(a_{\lfloor \cdot \rfloor})}; \]
  adopting the conventions ``$1/0 = +\infty$'' and ``$1/(+\infty) = 0$''.
 \end{lem}
 \begin{proof}
  We will only show that $\mlow(A) = 1/\mupp(a_{\lfloor \cdot\rfloor})$. The argument for the other equation is entirely analogous, but with opposite inequalities. First, suppose $\mlow(A)>0$. Taking $0< \gamma < \mlow(A)$, there exists $m>0$ and some $x_0 \in\R_{+}$ for which
  \[ \frac{A(y)}{y^{\gamma}} \geq m\frac{A(x)}{x^{\gamma}}, \quad \forall y\geq x\geq x_0. \]
  Hence, there must be some $n_0 \in \N$ such that
  \[ \frac{A(a_N)}{a_{N}^{\gamma}} \geq m\frac{A(a_n)}{a_n^{\gamma}}, \quad \forall N\geq n\geq n_0; \]
  but since $A(a_k) = k+1$,
  \[ \frac{a_N}{(N+1)^{1/\gamma}} \leq \frac{1}{m^{1/\gamma}} \frac{a_n}{(n+1)^{1/\gamma}}, \quad \forall N\geq n\geq n_0; \]
  which means $a_n/n^{1/\gamma}$ is almost decreasing, thus $\mupp(a_{\lfloor \cdot \rfloor}) \geq 1/\mlow(A)$.
   
  To see that strict inequality cannot hold, we now take $\gamma > \mlow(A)$. Note that now we are considering the possibility of having $\mlow(A) = 0$, for the next argument will also apply to the ``$0,+\infty$'' case. For this chosen $\gamma$, the function $A(x)/x^{\gamma}$ will not be almost increasing; that is, $\forall \eps>0,\forall M\in \R_{+}$ there are $y_{\eps} > x_{\eps} \geq M$ such that
  \[ \frac{A(y_{\eps})}{y_{\eps}^{\gamma}} \leq \eps \frac{A(x_\eps)}{x_{\eps}^{\gamma}}. \]
  Letting $n_\eps:= A(x_\eps)$ and $N_\eps := A(y_\eps)+1$, the above inequality implies
  \[ \frac{a_{N_\eps}}{(N_\eps-1)^{1/\gamma}} \geq \frac{1}{\eps^{1/\gamma}} \frac{a_{n_\eps}}{n_{\eps}^{1/\gamma}}, \]
  which means $a_n/n^{1/\gamma}$ cannot be almost decreasing, thus $\mupp(a_{\lfloor \cdot\rfloor}) < 1/\gamma$. Hence equality must hold when $\mlow(A) > 0$, and $\mlow(A)$ vanishes iff $\mupp(a_{\lfloor \cdot\rfloor}) = +\infty$.
 \end{proof}
 
 \begin{lem}\label{abab2}
  For every sequence $\A\subseteq\N$ and every $h\geq 1$,
  \begin{equation*}
    \mlow(s_{\A,h}) = h\cdot \mlow(A), \qquad \mupp(s_{\A,h}) = h\cdot \mupp(A).
  \end{equation*}
 \end{lem}
 \begin{proof}
  As in Lemma \ref{abab}, we will only show the $\mlow$ case, for the other is analogous but with opposite inequalities and ``$+\infty$'' instead of ``$0$''. First, suppose $\mlow(A)>0$. Assume for the sake of contradiction that $\mlow(s_{\A,h}) < h\cdot \mlow(A)$. This means that there is $\mlow(s_{\A,h}) < \gamma < h\cdot\mlow(A)$ for which there is $(x_k)_{k\in\N}$, $(y_k)_{k\in\N}$ with $y_k > x_k \to +\infty$ as $k\to +\infty$ such that
  \[ \frac{s_{\A,h}(y_k)}{s_{\A,h}(x_k)} < \frac{1}{k}\left(\frac{y_k}{x_k}\right)^\gamma. \]
  From Lemma \ref{ax2sax} we have
  \[ \frac{A(y_k/2)^h}{A(x_k)^h} \leq \frac{s_{\A,h}(y_k)}{s_{\A,h}(x_k)}, \]
  but, since $\gamma/h < \mlow(A)$, we conclude that $y_k/2 < x_k$ for all large $k$. Hence
  \[ \frac{s_{\A,h}(y_k)}{s_{\A,h}(x_k)} < \frac{2^\gamma}{k} \xrightarrow{k\to +\infty} 0, \]
  a contradiction. Therefore $\mlow(A) >0$ implies $\mlow(s_{\A,h}) \geq h\cdot \mlow(A)$.
  
  To see that strict inequality cannot hold, notice that $A(x)$ is always an increasing function and $A(x) \leq x+1$, hence $0\leq \mlow(A)\leq 1$. Take $\gamma > h\cdot \mlow(A)$, considering the possibility of having $\mlow(A)=0$. Similarly to the previous case, there is $(x_k)_{k\in\N}$, $(y_k)_{k\in\N}$ with $y_k > x_k \to +\infty$ as $k\to +\infty$ for which
  \[ \frac{A(y_k)}{A(x_k)} < \frac{1}{k^{1/h}}\left(\frac{y_k}{x_k}\right)^{\gamma/h}. \]
  From Lemma \ref{ax2sax} we have
  \[ \frac{s_{\A,h}(y_k)}{s_{\A,h}(x_k/2)} \leq \frac{A(y_k)^h}{A(x_k)^h}, \]
  therefore we may deduce that $\mlow(s_{\A,h}) \leq \gamma$, as required.
 \end{proof}

 Thus, we define a \emph{PI sequence} to be a sequence satisfying any of the following equivalent conditions.

 \begin{prop}[PI sequences] \label{chrpi}
  Let $\A\subseteq\N$ be a sequence. The following are equivalent:
  \begin{enumerate}[(i)]
   \item $A$ has positive increase;$\phantom{\liminf\limits_{x\to +\infty}}$ \label{a4i}
   \item $a_{2n} = O\left(a_n\right)$;$\phantom{\liminf\limits_{x\to +\infty}}$ \label{a4ii}
   \item $\liminf\limits_{x\to +\infty} A(\lambda x)/A(x) > 1$ for some $\lambda >1$; \label{a4iii}
   \item $s_{\A,h}$ has positive increase for all or at least some $h\geq 1$.$\phantom{\liminf\limits_{x\to +\infty}}$ \label{a4iv}
  \end{enumerate}
 \end{prop}
 \begin{proof}
  \eqref{a4i}$\iff$\eqref{a4ii} follows from Lemma \ref{abab} and the observation that a function is OR exactly when both its Matuszewska indices are finite, whereas \eqref{a4i}$\iff$\eqref{a4iv} follows from Lemma \ref{abab2}.  
  
  For \eqref{a4iv}$\implies$\eqref{a4iii}, notice that if $\mathcal{M}_{*}(s_{\A,h}) > 0$ for some $h\geq 1$, then for $0<\gamma<\mathcal{M}_{*}(s_{\A,h})$ there is $m>0$ and $x_0\in\R_{+}$ such that, for all $\mu > 1$,
  \[ \frac{s_{\A,h}(\mu x)}{s_{\A,h}(x)} \geq m\mu^{\gamma}, \quad \forall x\geq x_0. \]
  From Lemma \ref{ax2sax}, we know that
  \[ \frac{A(\mu x)^h}{A(x/2)^h} \geq \frac{s_{\A,h}(\mu x)}{s_{\A,h}(x)}, \]
  hence
  \[ \frac{A(2\mu x)}{A(x)} \geq m^{1/h} \mu^{\gamma/h}, \quad \forall x\geq x_0/2. \]
  Thus, if $\mu$ is large enough, we have $\liminf\limits_{x\to +\infty} A(\lambda x)/A(x) > 1$ for $\lambda = 2\mu$.  
  
  Finally, to show \eqref{a4iii}$\implies$\eqref{a4i}, assume that \eqref{a4i} is false. Then for every $\eps > 0$ there is a sequence $(n_k)_{k\in\N}$ with $n_k = n_k(\eps) \to +\infty$ as $k\to +\infty$ for which
  \[ \frac{a_{\lfloor (1+\eps)n_k \rfloor}}{a_{n_k}} > k. \]
  This is the same as $ka_{n_k} < a_{\lfloor (1+\eps)n_k\rfloor}$. Since $A(a_{k}) = k+1$, it follows that
  \[ \frac{A(ka_{n_k})}{A(\corr{a_{n_k}})} < \frac{\lfloor (1+\eps)n_k\rfloor + 1}{n_k + 1} \leq 1+\eps + \frac{1-\eps}{n_k +1}; \]
  that is, for every real number $\lambda > 1$,
  \begin{equation*}
   \frac{A(\lambda \corr{a_{n_k}})}{A(\corr{a_{n_k}})} < 1 + \eps + o(1).
  \end{equation*}
  as $k\to +\infty$. Since for every $\eps > 0$ there is such $(n_k)_{k\in\N}$, we conclude that for every $\lambda > 1$ it holds that $\liminf_{x\to +\infty} A(\lambda x)/A(x) = 1$, the negation of which is exactly item \eqref{a4iii}.
 \end{proof}

 \begin{rem}[Criteria for non-bases]\label{nnbss}
  We describe shortly another motivation for PI sequences. In Section 7 of St\"{o}hr \cite{stoh55} it is presented a list of criteria for non-bases, i.e. criteria which if a sequence $\A\subseteq\N$ satisfies it, then it cannot be an additive basis. Among these, we highlight the following two:
  \begin{itemize}
   \item \emph{Criterion \#1}: $A(x) \not\gg x^{\eps}$ for every $\eps>0$.
 
   \item \emph{Criterion \#2}: $\limsup\limits_{n\to+\infty} a_{n+k}/a_n = +\infty$ for some $k\geq 1$.
  \end{itemize}
  When studying bases, it is then natural to focus only on sequences that avoid such criteria. A simple way of doing this would be to consider a condition that simultaneously avoid the above two statements. We may then consider the following questions:
  \begin{enumerate}[(1)]
   \item What is the smallest $f$ for which if $A(n) \gg f(n)$ then $\A$ does not satisfy criterion \#2? \label{quest1}
 
   \item What is the smallest $f$ for which if $\limsup_{n\to+\infty} a_{n+f(n)}/a_n < +\infty$ then $\A$ does not satisfy criterion \#1? \label{quest2}
  \end{enumerate}
 
 For question (\ref{quest1}), it has been long known that a sequence $\A\subseteq \N$ with $A(n)=\Theta(n)$ is an additive bases if and only if it satisfies the trivial requirement $\{0,1\}\subseteq \A$; this is indeed equivalent to a celebrated theorem by Schnirelmann (cf. Subsection \ref{subsecA5}).
 
 \corr{For (\ref{quest2}), condition \ref{a4ii} of Proposition \ref{chrpi} says precisely that $a_{2n}=O(a_n)$ is equivalent to $A$ having positive increase. In particular, this implies that there is $\eps>0$ for which $A(x)/x^\varepsilon$ is almost increasing, and hence $A(x)\gg x^\varepsilon$.  More generally, a condition of the form $\limsup_{n\to\infty} a_{n+f(n)}/a_n <\infty$ with $f(n)\gg n$ implies the same conclusion. We note that} not all bases are necessarily PI (cf. Remark \ref{cnter}). An interesting line of research would be to identify which PI sequences are not bases.
 \end{rem}
 
\subsection{OR\texorpdfstring{$_+$}{\_+} sequences}\label{appA4}
 Finally we arrive at the central point of this appendix, which are sequences that are OR and PI at the same time. The reason behind considering these two conditions simultaneously comes from the following lemma, which is of central importance to this paper.

 \begin{lem}[OR--PI lemma] \label{pilem}
  Let $f:[x_0,+\infty)\to\R_+$ be a positive and locally integrable real function. The following are equivalent:
  \begin{enumerate}[(i)]
   \item $\displaystyle{\int_{x_0}^x \frac{f(t)}{t}\mathrm{d}t = \Theta(f(x))};$
   \item $\displaystyle{0<\mlow(f)\leq\mupp(f)<+\infty ; \phantom{\int_{\alpha}^x}}$
   \item $f$ is $O$-regularly varying and has positive increase.
  \end{enumerate}
 \end{lem}
 \begin{proof}
  Corollary 2.6.2, p. 96 of Bingham, Goldie \& Teugels \cite{bingham89}.
 \end{proof}

 We define an \emph{OR$_{+}$ sequence} to be a sequence satisfying any of the equivalent conditions of the following proposition. The ``$+$'' in OR$_{+}$ is used to indicate the positivity of both Matuszewska indices.
 
 \begin{prop}[OR$_{+}$ sequences] \label{chrorpi}
  Let $\A\subseteq\N$ be a sequence. The following are equivalent:
  \begin{enumerate}[(i)]
   \item $A$ is $O$-regularly varying and has positive increase;$\phantom{\displaystyle{\int_{1}^x}}$
   \item $\displaystyle{\int_{1}^x \frac{A(t)}{t}\mathrm{d}t = \Theta(A(x))}$;
   \item Both $A(2x) = O(A(x))$ and $a_{2n} = O(a_n)$; $\phantom{\displaystyle{\int_{1}^x}}$
   \item $\displaystyle{\int_{1}^x \frac{s_{\A,h}(t)}{t}\mathrm{d}t = \Theta(s_{\A,h}(x)})$ for all or at least some $h\geq 1$.
  \end{enumerate}
 \end{prop}
 
 The proof is just a straightforward combination of Proposition \ref{chrreg}, Proposition \ref{chrpi} and Lemma \ref{pilem}.
 
 \begin{rem}[Counterexamples] \label{cnter}
  Almost all common examples of additive bases are OR$_{+}$, such as the sequence of primes and sequences generated by polynomials, which even have regularly varying counting functions. Nonetheless, not all bases are common! For instance, the sequence
  \[ \mathscr{C} := [0,2^{2^{17}}] \cup \left(\bigcup_{k\geq 17} \left\{2^{2^k} + t(k-1) : 0\leq t\leq 2^{2^k}\right\} \cup [k\cdot 2^{2^k}, 2^{2^{k+1}}] \right) \]
  is an example of a $2$-basis, i.e. a basis of order $2$, which is not PI. It is a relatively simple exercise to show that $2\mathscr{C} = \N$; to see that it is not PI, let $n_k$ be such that $c_{n_k} = 2^{2^k}$. One can then show that $n_k \sim 2^{2^k}$ as $k\to +\infty$, and thus $c_{2n_k}/c_{n_k} = k-o(k)$, implying $c_{2n} \neq O(c_{n})$, which by the last lemma yields $\mlow(C) = 0$.
 
  On the other hand, the sequence 
  \[ \mathscr{D}:= \mathcal{P}_3\cup \left\{(2^{2^k})^3+t : 0\leq t\leq (2^{2^k})^2,~k\in\N \right\}, \]
  where $\mathcal{P}_3$ is the sequence of cubes, is not OR. It is a basis, for it contains $\mathcal{P}_3$, which is well-known to be a $9$-basis;\footnote{Wieferich--Kempner theorem (cf. Chapter 2 of Nathanson \cite{nathanson96}).} to see it is not OR, let $n_k = (2^{2^{k}})^3$. It is not hard to show that $D(n_k) = o(n_k^{2/3})$ as $k\to+\infty$, and then
  \[ \frac{D(2n_k)}{D(n_k)} \geq \frac{D(n_k + n_k^{2/3})}{D(n_k)} = 1 + \frac{n_k^{2/3}}{D(n_k)}, \]
  which diverges, implying $D(2x) \neq O(D(x))$ and yielding $\mupp(D) = +\infty$.
 \end{rem}
 
\subsection{An application to additive bases}\label{subsecA5}
 We finish this appendix by showing a relatively general version of what is sometimes loosely referred to in the literature as \emph{Schnirelmann's method}. This refers to the common element between the elementary proofs of the sequence of primes with $\{0,1\}$\footnote{Schnirelmann-Goldbach theorem (cf. Chapter 7 of Nathanson \cite{nathanson96}).} and the sequence of $k$-powers for all $k\geq 1$\footnote{Linnik's solution to Waring's problem (cf. Chapter 2 of Gel'fond \& Linnik \cite{gelfond66}).} constituting additive bases, with ``elementary'' meaning avoiding the complex analytical approach of the Hardy--Littlewood circle method. Although seemingly paradoxically, we believe that our idea is most concisely stated on the context of Vinogradov's form of the circle method. This part is of complete independent interest from the main subject of this paper, although some analogies can certainly be drawn as we shall point out. We now introduce the necessary concepts and notation.

 Fix $\A\subseteq\N$ a sequence. For $\alpha\in\R/\Z$, let $e(\alpha) := e^{2\pi i \alpha}$. This is well-defined, for the value of $e^{2\pi i x}$ for $x\in\R$ depends only on the fractional part of $x$. Consider then the following truncated generating-type function:
 \begin{equation}
  \mathcal{G}^{(n)}_{\A}(\alpha) := \sum_{k \leq n} \mathbbm{1}_{\A}(k) e(k\alpha). \label{genfe}
 \end{equation}
 The functions $\{e(k\alpha)\}_{k\in\Z}$ form an orthonormal basis of $L^2(\R/\Z)$. From this orthonormality relation it follows, similarly to the formal series in \eqref{forms}, 
 \begin{equation}
  r_{\A,h}(n) = \int_{\R/\Z} \mathcal{G}^{(n)}_{\A}(\alpha)^h e(-n\alpha) \mathrm{d}\alpha. \label{mthcrc}
 \end{equation}
 This is the starting point of Vinogradov's form of the circle method. Following mainly Chapters 5 and 8 of Nathanson \cite{nathanson96}, we describe something which may be thought of as a general framework of the method. In both Waring's and Goldbach's ternary representation problem, a similar sort of \emph{ansatz} is employed. The idea is to first consider the following ``averaged'' form of \eqref{genfe}: 
 \begin{equation}
  \widetilde{\mathcal{G}}^{(n)}_{\A}(\alpha) := \sum_{k \leq n} \frac{A(k)}{k+1}e(k\alpha). \label{gtilenfe}
 \end{equation}
 This is similar to what we did in the construction of $\mathcal{S}_\A$. Even though we are not necessarily requiring $\A$ to be OR$_+$, our goal is to show that the OR$_+$ condition fits naturally into this framework. Indeed, the next step is to approximate \eqref{mthcrc} by 
 \begin{equation}
  J_{\A,h}(n) := \int_{\R/\Z} \widetilde{\mathcal{G}}^{(n)}_{\A}(\alpha)^h e(-n\alpha) \mathrm{d}\alpha, \label{jdef}
 \end{equation}
 the so-called \emph{singular integral} in the context of the usual circle method. Again, the orthonormality of $\{e(k\alpha)\}_{k\in\Z}$ implies that 
 \begin{equation}
  J_{\A,h}(n) = \sum_{\substack{k_1+\ldots +k_h = n \\ k_1,\ldots,k_h \geq 0}} \frac{A(k_1)}{k_1+1}\cdots \frac{A(k_h)}{k_h + 1}. \label{orbtt}
 \end{equation}
 When $\A$ is OR$_+$ one can show that $J_{\A,h}(n) = \Theta(A(n)^h/n)$ for all $h\geq 1$; this is exactly equivalent to Lemma \ref{t37} in the main text, which we called the \emph{fundamental lemma}. Finally, one aims to arrive at
 \begin{equation}
  r_{\A,h}(n) = \mathfrak{S}_{\A,h}(n)J_{\A,h}(n) + \text{error}, \label{crcmth}
 \end{equation}
 where $\mathfrak{S}_{\A,h}$  is some arithmetic function and the error term is at most $o(r_{\A,h}(n))$. It is worth noting that \emph{every sequence} $\A\subseteq\N$ has such a representation for $r_{\A,h}$ for every $h\geq 1$; i.e. we can always choose $\mathfrak{S}_{\A,h}$ such that $r_{\A,h}$ satisfies \eqref{crcmth}, the trivial choice being $\mathfrak{S}_{\A,h}(n) = r_{\A,h}(n)/J_{\A,h}(n)$ when the denominator is greater than $0$, and $\mathfrak{S}_{\A,h}(n) = 1$ otherwise. Furthermore, by the way we defined the error term, it is always the case that
 \begin{equation}
  \mathfrak{S}_{\A,h}(n) \sim \frac{r_{\A,h}(n)}{J_{\A,h}(n)} \text{ as } n\to +\infty \text{ through } h\A. \label{sahroj}
 \end{equation}
 
 \begin{rem}[Relation to Main Theorem]
  One way to interpret the conclusion after Theorem \ref{p314} is that, for suitable $\A$, in $\mathcal{S}_{\A}$ one has $\mathfrak{S}_{\omega,h}(n) \uptxt{a.s.}{=} \Theta(1)$.
 \end{rem}
 
 In both Waring's and Goldbach's problem, the idea of the method centers around finding a good description to $\mathfrak{S}_{\A,h}$ (the so-called \emph{singular series}) and controlling the error term. For that aim, one partitions $\R/\Z$ into \emph{major}  and \emph{minor arcs} for each sufficiently large $n$, usually denoted by $\mathfrak{M} = \mathfrak{M}(n)$ and $\mathfrak{m} = \mathfrak{m}(n)$ resp.. This partition is constructed in a way such that the integral in \eqref{mthcrc} over the minor arcs constitutes part of the error term, while over the major arcs one seeks for a representation of $r_{\A,h}$ as in \eqref{crcmth}.
 
 The method as generally employed, however, draws heavily upon the number-theoretical properties of the sequences being considered, demonstrated by the use of Hua's lemma and Weyl's inequality in the case of Waring, and specific estimates of exponential sums over primes together with non-trivial estimates for the error term in Dirichlet's prime number theorem in the case of ternary Goldbach.
 
 If on the other hand we dispense with the precision achieved by the method in the case of highly number-theoretical sequences and focus only on the more relaxed property of just being an additive basis, we find that through OR$_+$ sequences it is possible to marry our description of the circle method with the elementary approach of Schnirelmann's method. To properly state this result, we only need two more things. The first one was already mentioned in Remark \ref{nnbss}.
 
 \theoremstyle{plain}
 \newtheorem*{st}{Schnirelmann's theorem}
 \begin{st}[Theorem 4, p. 8 of Halberstam \& Roth \cite{halberstam83}]
  If a sequence $\A\subseteq\N$ contains $\{0,1\}$ and satisfies $A(x)=\Theta(x)$, then $\A$ is an additive basis.
 \end{st}
 
 The second one is the notion of \emph{stability} in additive bases as described in Section 11.2 of Nathanson \cite{nathanson99}. A sequence $\A\subseteq\N$ is said to be a \emph{stable basis} when $\A$ is a basis and every subsequence $\B\subseteq\A$ with $\{0,1\}\subseteq \B$ and $B(x)= \Theta(A(x))$ is also a basis. Without further ado:
 
 \begin{thm}\label{schnmth}
  Let $\A\subseteq\N$ be an OR$_+$ sequence with $\{0,1\}\subseteq \A$. If there is $\eps > 0$ such that
  \begin{equation}
   \frac{1}{x}\sum_{n\leq x} \mathfrak{S}_{\A,h}(n)^{1+\eps} = O(1) \label{schn}
  \end{equation}
  in \eqref{crcmth} for some $h\geq 1$, then $\A$ is a stable additive basis.
 \end{thm}
 \begin{proof}
  From equations \eqref{orbtt} and \eqref{sahroj}, the OR$_+$ condition implies that \eqref{schn} is equivalent to $\sum_{n \leq x} (n\cdot r_{\A,h}(n)/A(n)^h)^{1+\eps} = O(x)$. Recall that $\A$ is, in particular, PI. By item \eqref{a4iii} of Proposition \ref{chrpi}, there is $\lambda > 1$ and $\delta > 0$ for which $A(\lambda x)/A(x) > 1+\delta$ for all large $x$, therefore
  \begin{align*}
   A(x)^{(1+\eps)h/\eps} &< \delta^{-(1+\eps)h/\eps} (A(\lambda x) - A(x))^{(1+\eps)h/\eps} \\
   &= \delta^{-(1+\eps)h/\eps} \left(\sum_{x < n \leq \lambda x} \mathbbm{1}_{\A}(n)\right)^{(1+\eps)h/\eps} \\
   &\leq \delta^{-(1+\eps)h/\eps} \left(\sum_{hx < n \leq h\lambda x} r_{\A,h}(n) \right)^{(1+\eps)/\eps}.
  \end{align*}
  We know that the sequence $\A$ is also, in particular, OR. Therefore, by H\"{o}lder's inequality,\footnote{Theorem A in Section 42, p. 175 of Halmos \cite{halmos74}.} since $1 = 1/(1+\eps) + \eps/(1+\eps)$,
  \begin{align*}
   \Bigg(&\sum_{hx < n \leq h\lambda x} r_{\A,h}(n) \Bigg)^{(1+\eps)/\eps} \leq |h\A\cap [hx, h\lambda x]| \left(\sum_{hx < n \leq h\lambda x} r_{\A,h}(n)^{1+\eps}\right)^{1/\eps} \\
   &\leq |h\A\cap [hx, h\lambda x]| \left(\frac{A(h\lambda x)^{h}}{hx} \right)^{(1+\eps)/\eps} \left(\sum_{hx < n \leq h\lambda x} \left(n\frac{r_{\A,h}(n)}{A(n)^h}\right)^{1+\eps} \right)^{1/\eps} \\
   &\ll |h\A\cap [hx, h\lambda x]| \frac{A(x)^{(1+\eps)h/\eps}}{x},
  \end{align*}
  hence $|h\A\cap [0, x]| = \Theta(x)$. By Schnirelmann's theorem, we deduce that $\A$ is a basis. To see that $\A$ is a stable basis, notice that if $\B\subseteq \A$ then $r_{\B,h}(n)\leq r_{\A,h}(n)$, thus $B(x) \asymp A(x)$ implies $\mathfrak{S}_{\B,h}(n) \ll \mathfrak{S}_{\A,h}(n)$. Hence, repeating the same argument for $\B$, since $\{0,1\}\subseteq \B$, we apply Schnirelmann's theorem again.
 \end{proof}
 
 A similar version for asymptotic bases can be deduced from the appropriate variant of Schnirelmann's theorem. A sequence $\A\subseteq\N$ is called an \emph{asymptotic basis} when there is $h\geq 2$ such that $\N\setminus h\A$ is finite, and from Theorem 11.6, p. 366 of Nathanson \cite{nathanson99} one knows that, if $A(n) = \Theta(n)$, then not being contained in some non-trivial arithmetic progression ensures that $\A$ is an asymptotic basis.
 
 \begin{rem}[Applications of Theorem \ref{schnmth}]
  The connection with elementary methods comes from \eqref{sahroj}, as both the proof of Schnirelmann--Goldbach's theorem and Linnik's elementary solution to Waring's problem is achieved by giving non-trivial upper bounds to $r_{\A,h}(n)/s_{\A,h}(n)$ for some large enough $h\geq 2$.
  
  Following Chapter 11 of Nathanson \cite{nathanson99}, let $f\in\Q[x]$ be \corr{a nonconstant} integer-valued polynomial (i.e. $f(k)\in \Z$ when $k\in\Z$) with positive leading coefficient. Define $\mathcal{P}_f := \{f(n) : n\in\N\}\cap \N$, the sequence generated by $f$. Since $|\mathcal{P}_f\cap[0,x]| = \corr{\Theta(x^{1/\deg(f)})}$, we know that $\mathcal{P}_f$ is OR$_{+}$. To show \corr{$\mathcal{P}_f\cup\{0,1\}$} is a stable basis, one proves that there is $h=h_f$ for which $\mathfrak{S}_{\mathcal{P}_f, h}(n) = O(1)$.\footnote{Theorem 11.8, p. 370 of Nathanson \cite{nathanson99}.}
 
  In the case of primes, the weak form of the Prime Number Theorem says that $\pi(x) = \Theta(x/\log(x))$, where $\pi(x) := |\mathbb{P}\cap[0,x]|$, the sequence of primes; hence $\mathbb{P}$ is OR$_{+}$. At the heart of Schnirelmann's proof that $\mathbb{P}\cup\{0,1\}$ forms a stable basis is the fact that $\mathfrak{S}_{\mathbb{P},2}(n) = O(\prod_{p\mid n}(1+p^{-1}))$,\footnote{Theorem 7.2, p. 186 of Nathanson \cite{nathanson96}.} and one can then show that $\mathbb{P}$ satisfies \eqref{schn} with $\eps = 1$. For the reader familiar with algebraic number theory, one can take $\mathbb{P}_K$ to be the set of primes that split completely in a number field $K/\Q$. From Chebotarev's density theorem\footnote{Theorem 13.4 in Chapter VII, p. 545 of Neukirch \cite{neukirchEN}.} one has $\pi_K(x) \asymp \pi(x)$, where $\pi_K(x) := |\mathbb{P}_K\cap [0,x]|$, therefore $\mathbb{P}_K\cup\{0,1\}$ is also a basis.
 \end{rem}
 
 \begin{rem}\label{contra2}
  As an interesting side note, it is possible to show that when $\A$ is OR$_{+}$, we have $\sum_{n\leq x} \mathfrak{S}_{\A,h}(n) = O(x)$ for all $h\geq 1$, which is the same of having $\eps = 0$ in \eqref{schn}. Since this fact is not essential to Theorem \ref{schnmth}, we chose to omit the proof, which is essentially a lengthy calculation.  
 \end{rem}

% ----------------------------------------------------------------
\bibliographystyle{amsplain}
\bibliography{$HOME/Academie/Recherche/_latex/bibliotheca}%

\providecommand{\bysame}{\leavevmode\hbox to3em{\hrulefill}\thinspace}
\providecommand{\MR}{\relax\ifhmode\unskip\space\fi MR }
\providecommand{\MRhref}[2]{%
  \href{http://www.ams.org/mathscinet-getitem?mr=#1}{#2}
}
\providecommand{\href}[2]{#2}
\begin{thebibliography}{10}

\bibitem{alon16}
N.~Alon and J.~H. Spencer, \emph{The probabilistic method}, 4th ed., John Wiley
  \& Sons, New Jersey, 2016.

\bibitem{bingham89}
N.~H. Bingham, C.~M. Goldie, and J.~L. Teugels, \emph{Regular variation},
  Cambridge Univ. Press, 1989.

\bibitem{bopspe89}
R.~B. Boppana and J.~H. Spencer, \emph{A useful elementary correlation
  inequality}, J. Comb. Theory A \textbf{50} (1989), 305--307.

\bibitem{erdo56}
P.~Erd\H{o}s, \emph{Problems and results in additive number theory}, Colloque
  sur la Theorie des Nombres (CBRM) (Bruxelles), 1956, pp.~127--137.

\bibitem{erdren60}
P.~Erd\H{o}s and A.~R\'{e}nyi, \emph{Additive properties of random sequences of
  positive integers}, Acta Arith. \textbf{6} (1960), 83--110.

\bibitem{erdtet90}
P.~Erd\H{o}s and P.~Tetali, \emph{Representations of integers as the sum of $k$
  terms}, Random Structures Algorithms \textbf{1} (1990), 245--261.

\bibitem{feller1}
W.~Feller, \emph{An introduction to probability theory and its applications},
  3rd ed., vol.~1, John Wiley \& Sons Inc., New York, 1967.

\bibitem{feller2}
\bysame, \emph{An introduction to probability theory and its applications}, 2nd
  ed., vol.~2, John Wiley \& Sons Inc., New York, 1971.

\bibitem{gelfond66}
A.~O. Gel'fond and Yu.~V. Linnik, \emph{Elementary methods in the analytic
  theory of numbers}, Pergamon Press, Oxford, 1966, D. E. Brown translated
  version.

\bibitem{halberstam83}
H.~Halberstam and K.~F. Roth, \emph{Sequences}, revised ed., Springer, 1983.

\bibitem{halmos74}
P.~R. Halmos, \emph{Measure theory}, Graduate Texts in Mathematics, vol.~18,
  Springer, 1974.

\bibitem{kol95}
M.~N. Kolountzakis, \emph{An effective additive basis for the integers},
  Discrete Math. \textbf{145} (1995), 307--313.

\bibitem{nathanson96}
M.~B. Nathanson, \emph{Additive number theory: The classical bases}, 2nd ed.,
  Graduate Texts in Mathematics, vol. 164, Springer, 1996.

\bibitem{nathanson99}
\bysame, \emph{Elementary methods in number theory}, Graduate Texts in
  Mathematics, vol. 195, Springer, 1999.

\bibitem{nat12}
\bysame, \emph{Thin bases in additive number theory}, Discrete Math.
  \textbf{312} (2012), 2069--2075.

\bibitem{neukirchEN}
J.~Neukirch, \emph{Algebraic number theory}, Grundlehren der mathematischen
  {W}issenschaften, vol. 322, Springer, 1999.

\bibitem{stoh55}
A.~St\"{o}hr, \emph{{G}el\"{o}ste und ungel\"{o}ste {F}ragen \"{u}ber {B}asen
  der nat\"{u}rlichen {Z}ahlenreihe {I}}, J. reine angew. Math. \textbf{194}
  (1955), 40--65.

\bibitem{taf25}
C.~T\'{a}fula, \emph{Representation functions with prescribed rates of growth},
  Random Structures Algorithms \textbf{69} (2026), no.~1, e70088.

\bibitem{taovu06}
T.~Tao and V.~H. Vu, \emph{Additive combinatorics}, Cambridge Stud. Adv. Math.,
  vol. 105, Cambridge Univ. Press, 2006.

\bibitem{vvu00wp}
V.~H. Vu, \emph{On a refinement of {W}aring's problem}, Duke Math. J.
  \textbf{105} (2000), 107--134.

\bibitem{vvu00mc}
\bysame, \emph{On the concentration of multivariate polynomials with small
  expectation}, Random Structures Algorithms \textbf{16} (2000), 344--363.

\bibitem{war17}
L.~Warnke, \emph{Upper tails for arithmetic progressions in random subsets},
  Israel J. Math. \textbf{221} (2017), 317--365.

\end{thebibliography}
\end{document}